\documentclass[11pt,a4paper]{article}
\usepackage[USenglish]{babel}
\usepackage{amssymb}
\usepackage{amsfonts}
\usepackage{amsmath}
\usepackage{amsthm}
\usepackage{commath}
\usepackage{graphicx}
\usepackage[usenames, dvipsnames]{xcolor}
\usepackage{dsfont}
\usepackage{verbatim}
\usepackage{xcolor}
\definecolor{myblue}{cmyk}{0.5, 0.1, 0.1, 0.1}
\usepackage{csquotes}
\usepackage{enumitem} 
\usepackage{lmodern}
\usepackage{bm}
\usepackage[labelfont=bf]{caption}

\usepackage[numbers]{natbib}
\setcitestyle{aysep={ }} 

\graphicspath{ {"/home/fisibubele/Documents/Studium/TU Dresden/Notes/publication/submission/arXiv/v2"} }

\allowdisplaybreaks[4]

\numberwithin{equation}{section}

\newtheorem{theorem}{Theorem}[section]
\newtheorem{lemma}[theorem]{Lemma}

\theoremstyle{definition}

\newtheorem{example}[theorem]{Example}
\theoremstyle{remark}
\newtheorem{remark}[theorem]{Remark}
\newtheorem{condition}{Condition}[section]

\newcommand*\diff{\mathop{}\!\mathrm{d}}

\newcommand{\cC}{\mathcal{C}}

\newcommand{\EE}{\mathbb{E}}

\newcommand{\PP}{\mathbb{P}}

\newcommand{\cB}{\mathcal{B}}
\newcommand{\cA}{\mathcal{A}}

\newcommand{\cF}{\mathcal{F}}
\newcommand{\cD}{\mathcal{D}}

\newcommand{\supp}{\mbox{\rm supp}\,}
\newcommand{\NN}{\mathbb{N}}

\newcommand{\RR}{\mathbb{R}}

\newcommand{\dd}{\mathop{}\!\mathrm{d}}

\begin{document}
\author{Anita Behme\thanks{Technische Universit\"at Dresden,
		Institut f\"ur Mathematische Stochastik, 01062 Dresden, Germany, e-mail: anita.behme@tu-dresden.de \& claudius.luetke\_schwienhorst@tu-dresden.de} 
		\,\,and Claudius Lütke Schwienhorst\footnotemark[1]\,\,\textsuperscript{,}\thanks{Corresponding author.}} 
\title{Lévy Langevin Monte Carlo for sampling from heavy-tailed target distributions}
\maketitle

\begin{abstract}
We extend the Lévy Langevin Monte Carlo method studied by \cite{oechsler-levy-2024}: Choosing a heavy-tailed target distribution we prove convergence of a solution of a stochastic differential equation to this target. Hereby, the stochastic differential equation is driven by a compound Poisson process - unlike in the case of a classical Langevin diffusion. The method allows one to sample from non-smooth targets and distributions with separated modes with exponential convergence to the invariant distribution, which in general cannot be guaranteed by the classical Langevin diffusion in presence of heavy tails. The method is promising due to the possibility of a simple implementation because of the compound Poisson noise term.
\end{abstract}

\noindent
{\em AMS 2010 Subject Classifications:} \, primary:\,\,\, 60G51 · 60H10 · 60G10 \,\,\,
secondary: \,\,\, 65C05

\noindent
{\em Keywords:}
	Exponential convergence, Heavy-tailed target distribution, Invariant distributions, Langevin Monte Carlo, Lévy process, Markov Chain Monte Carlo, Stochastic differential equation

\section{Introduction}

Producing samples from a probability distribution is an important task in science and in particular for many machine learning algorithms. Having access to random samples of a probability distribution is relevant in various ways, for example for calculation of moments and high-dimensional integration in general, or uncertainty quantification in Bayesian statistics, just to name two, cf. \cite[Chapter 8]{chopin-introduction-2020} or \cite[Chapter 6]{gelman-bayesian-2013}.

Markov Chain Monte Carlo methods achieve this by constructing a Markov chain that has the desired distribution as invariant measure. If the chain is ergodic towards this measure one obtains random samples by simulating the chain forward in time. This work builds upon a particular instance of a Markov Chain Monte Carlo method which uses Langevin dynamics to produce samples. Consider a probability density $\pi$ on $\mathbb{R}^d$ and a Brownian motion $(W_t)_{t\geq 0}$ on the filtered probability space $(\Omega,\cF, (\mathcal{F}_t)_{t\geq 0},\PP)$. It is then well known that under appropriate conditions on $\pi$ the strong solution of the stochastic differential equation (SDE)
\begin{equation} \label{langevin diffusion}
	\dd X_t = \frac{\nabla\pi (X_t)}{\pi(X_t)} \dd t + \dd W_t,\quad t\geq 0, 
\end{equation}
converges towards its unique invariant measure which is given by $\bm{\pi}(\dd x)=\pi(x)\dd x$. We call the algorithm associated to the discretisation (via an Euler scheme) of this SDE \emph{Langevin Monte Carlo} (LMC).

In \cite{oechsler-levy-2024}, the author posed the question whether one can consider a more general version of the SDE \eqref{langevin diffusion} of the form
\begin{equation} \label{levy langevin diffusion}
	\dd X_t = \phi(X_t)\dd t + \dd L_t, \quad t\geq 0,
\end{equation}
where $L=(L_t)_{t\geq 0}$ is a general Lévy process and $\phi$ an appropriate drift function. It was shown in \cite{oechsler-levy-2024}, exploiting results w.r.t. invariant measures from \cite{behme-invariant-2024}, that $\phi$ and $L$ can be chosen such that the strong solution to \eqref{levy langevin diffusion} converges in total variation norm to the target measure. As in \cite{oechsler-levy-2024}, we call the algorithm associated to the discretisation of this SDE \emph{L\'evy Langevin Monte Carlo} (LLMC).

LLMC allows to overcome several weaknesses of the classical algorithm using LMC. In \cite{oechsler-levy-2024}, the focus has been on the setting where the target density is non-smooth or has separated modes. It is known that under these assumptions the continuous time solution of \eqref{langevin diffusion} already poses problems in terms of applicability due to the non-differentiability or very slow convergence speed and poor mixing. Beyond that, the Euler discretisation of this solution then inherits these difficulties. Another well-known weakness of the solution to \eqref{langevin diffusion} is its slow convergence to the target distribution in the presence of heavy tails, and again this is inherited by its discretisation. Therefore in this article we generalise the Lévy Langevin Monte Carlo method to allow for efficient sampling from heavy-tailed distributions.

\subsection{Relation to literature}

There is an enormous amount of literature on sampling algorithms based on the Langevin diffusion \eqref{langevin diffusion} and Markov chain Monte Carlo (MCMC) methods. The strengths of MCMC methods in general lie in the fact that they perform better than direct methods like inverse transform sampling or rejection sampling whenever distributions are more complex (e.g. multimodal) or high-dimensional. We refer to \cite{brooks-handbook-2011} and \cite{roberts-exponential-1996} for general overviews on MCMC and Langevin sampling.

The idea to generalise the Langevin diffusion to Lévy(-type) noise is not new. Works of different scope can be found in the literature. In \cite{simsekli-fractional-2017}, a sampling algorithm for which $(L_t)_{t\geq 0}$ is a stable process is put forward and further analysed in \cite{nguyen-non-asymptotic-2019}, in particular w.r.t. applications in non-convex optimisation. In \cite{eliazar-levy-2003}, various examples of Lévy-driven SDEs are discussed, in particular w.r.t. their stationary distributions. These works are similar with regard to their focus on practical aspects and their neglect of a rigorous theoretical analysis of the SDE, its invariant measure and ergodic behaviour. More precisely, \cite{simsekli-fractional-2017} and \cite{eliazar-levy-2003} disregard the fact that in general there is no equivalence between the stationary solution of a Fokker-Planck equation and the invariant measure of the associated Markov process, not to mention the question of uniqueness of a stationary solution of the Fokker-Planck equation.

In \cite{huang-approximation-2020}, the authors fill up the missing theoretical analysis in \cite{simsekli-fractional-2017} by carefully deriving assumptions on the coefficient of the stable process driven SDE and proving the ergodicity towards the unique invariant measure. Nevertheless, they leave unanswered which precise numerical scheme of the continuous time SDE would be employed. This is especially important in their setup, since it would need a discussion of the approximation of the drift term given by a fractional Laplacian and the noise given by a stable process.

In \cite{zhang-ergodicity-2023}, a generalised version of the SDE \eqref{levy langevin diffusion} is put forward, namely
\begin{equation*}
	\dd X_t = \phi(X_t)\dd t + \sigma(X_t)\dd L_t, \quad t\geq 0,
\end{equation*}
where $(L_t)_{t\geq 0}$ is a stable process. Again, the practical implementation of a sampling algorithm remains unclear, in particular the aforementioned approximation of drift and noise. In a sense, this work can be seen as a refinement of \cite{huang-approximation-2020}, since the chosen SDE allows for a non-constant noise coefficient and simpler to check conditions on the coefficients are given.

Both articles, \cite{huang-approximation-2020} and \cite{zhang-ergodicity-2023}, work out the rigorous mathematical arguments needed for the ideas presented in \cite{simsekli-fractional-2017}. They agree with the work of this text in using the Foster-Lyapunov criteria developed in \cite{meyn-stability-1992,meyn-stability-1993,meyn-stability-1993-1} to conclude ergodicity. Still, they differ to our approach with regards to the argument for a unique limiting distribution: Instead of showing the Feller property as done in \cite{huang-approximation-2020, zhang-ergodicity-2023}, here we use the distributional equation derived in \cite{behme-invariant-2024}. It has to be said, referring to \cite[Thm. 3.37]{liggett-continuous-2010}, that both articles fail to finalise their argument via the Feller property, since a rigorous argument why the test functions form a core of the generator is missing. 

Apart from these technicalities, our approach yields two advantages: First, in \cite{huang-approximation-2020} and \cite{zhang-ergodicity-2023}, the target density is assumed to be more regular than in this text in order to obtain greater regularity for the drift coefficient. Second, the driving noise is assumed to be a stable process unlike a compound Poisson process as done in this work. For a practical implementation this means: While the stable noise term must be approximated, the compound Poisson process can be simulated without truncation and results in a drift that in principle allows for Monte Carlo approximations due to finite-intensity jumps.

\subsection{Outline of this paper}

The article is structured as follows. In Section \ref{main}, we  begin with a specification of the SDE \eqref{levy langevin diffusion} in terms of the choice of drift coefficient and Lévy noise term and explain how this is justified according to \cite{behme-invariant-2024}.

In Section \ref{assumptions}, we state assumptions on the target distribution and the driving noise term. In Section \ref{Sec_heavy_tailed}, before stating the main result of this article, we introduce common classes of heavy-tailed distributions that are covered by it. Subsequent to the main theorem we elaborate in detail how the process has to be tuned in order to yield samples from these classes. In {Section \ref{Sec_simulate}} we conclude the main section with several simulations as a proof of concept.

Section \ref{Sec_proofs} contains the proof of the main result. Here, the argumentation will be split up in results concerning irreducibility and ergodicity of the approximating Markov process and why the unique invariant measure is actually the target distribution. We want to emphasise that Section \ref{uniqueness} describes a strategy to exploit the possibility to choose a non-smooth target density compared to the Langevin diffusion \eqref{langevin diffusion}. This choice is possible because the drift coefficient has a different form, but a different argument is then needed to show that the unique invariant measure of the Markov process is indeed the target distribution, since in general the solution to \eqref{levy langevin diffusion} lacks the Feller property.

\section{Main results} \label{main}

Let $(\Omega,\cF, (\mathcal{F}_t)_{t\geq 0},\PP)$ be a filtered probability space and assume that $(\mathcal{F}_t)_{t\geq 0}$ fulfills the usual assumptions of right continuity and completeness. Let $\pi$ be a probability density function on $(0,\infty)$ with $\pi>0$ almost everywhere. Further let $L=(L_t)_{t\geq 0}$ be a compound Poisson process with rate $1$ and jump distribution $\mu$ on $(0,\infty)$ adapted to $(\cF_t)_{t\geq 0}$.
We consider the strong solution to the SDE \eqref{levy langevin diffusion} with this noise term. Our choice of noise term implies that the solution process is a concatenation of the deterministic drift given by the solution $u(x,t)$ to the ordinary differential equation (ODE)
\begin{equation} \label{ode}
	\begin{cases}
		\frac{\dd}{\dd t}u(x,t)=\phi(u(x,t)) \\
		u(x,0) = x
	\end{cases}
\end{equation}
and the jumps induced by $L$.  We set as drift coefficient
\begin{equation} \label{drift}
	\phi(x):=-\frac{\overline{F}_\mu\ast\pi(x)}{\pi(x)}\mathds{1}_{(0,\infty)}(x), \quad x\in \RR,
\end{equation}
where $\overline{F}_\mu(x):=\mu((x,\infty))$ is the tail function of the jump distribution of $L$ and $\ast$ denotes the standard convolution of two functions. We denote by $F_\mu$ the distribution function associated to $\mu$. The assumptions specified below in Subsection \ref{assumptions} imply that \eqref{ode} has a unique solution for any initial value $x\in (0,\infty)$, and by the observation that the solution to \eqref{levy langevin diffusion} is a concatenation of this deterministic path and the process $L$ we conclude that it has a unique strong solution $(X_t)_{t\geq 0}$. As such it is a Markov process, see for example \cite[Thm. 6.4.5]{applebaum-levy-2009}.

Before being able to explain the choice of $\phi$ we shortly recall the concept of the infinitesimal generator of a Markov process and introduce the two different notions of invariance that are relevant for this study. We write  $\mathcal{C}_0((0,\infty))$ for the class of continuous functions on $(0,\infty)$ vanishing at infinity, and  $\mathcal{C}_c^\infty((0,\infty))$ for the class of infinitely often continuously differentiable functions with compact support defined on $(0,\infty)$, a class that we refer to as \emph{test functions}.

Let $(\Omega, \cF, (\cF_t)_{t\geq 0}, X=(X_t)_{t\geq 0}, \PP^x)$ be a universal Markov process with state space $(0,\infty)$ that is normal, i.e. such that $\PP^x(X_0=x)=1$ for all $x>0$, and denote the expectation w.r.t. $\PP^x$ by $\EE^x$. Then the \emph{pointwise infinitesimal generator} of $X$ is the pair $(\mathcal{A},\mathcal{D}(\mathcal{A}))$ defined by
\begin{equation} \label{pointwise generator}
	\mathcal{A}f(x):=\lim_{t\to 0}\frac{1}{t}\left(\EE^x\left(f(X_t)\right)-f(x)\right),\quad x\in (0,\infty),f\in\mathcal{D}(\mathcal{A}),
\end{equation}
where
\begin{equation*}
	\mathcal{D}(\mathcal{A}):=\bigg\{f\in\mathcal{C}_0((0,\infty))\,\bigg| \,\lim_{t\to 0}\frac{1}{t}\left(\EE^x\left(f(X_t)\right)-f(x)\right)\,\text{exists for all}\,x\in (0,\infty)\bigg\}.
\end{equation*}
The pointwise infinitesimal generator is an extension of the mostly used \emph{strong infinitesimal generator}, which is defined by a uniform limit in \eqref{pointwise generator}.

A probability measure $\nu$ on $(0,\infty)$ is called an \emph{invariant distribution} for the Markov process $X$ if 
\begin{equation*}
	\int_{(0,\infty)} \PP^x(X_t\in B) \nu(\dd x) = \nu(B)
\end{equation*}
for all $t\geq 0$ and $B\in\mathcal{B}((0,\infty))$, the Borel $\sigma$-algebra on $(0,\infty)$.
Further, a probability measure $\nu$ is \emph{infinitesimally invariant} for $X$ (or its associated generator $\cA$) if for all $f\in\mathcal{C}_c^\infty((0,\infty))$ it holds
\begin{equation*}
	\int_0^\infty \mathcal{A}f(x)\nu(\dd x) = 0,
\end{equation*}
where $\mathcal{A}$ is the pointwise infinitesimal generator of $X$.

As shown e.g. in \cite{behme-invariant-2024},  if we consider the SDE \eqref{levy langevin diffusion} with $L$ a compound Poisson process, an application of It\^o's formula and the Lévy-It\^o decomposition implies that the infinitesimal generator of the solution process of \eqref{levy langevin diffusion} has the form 
\begin{equation} \label{Levy-type operator}
	\mathcal{A}f(x)=\phi(x)f^\prime(x) + \int_0^\infty \left(f(x+z)-f(x)\right)\mu (\dd z),
\end{equation}
where the first term is associated to the drift and the second to the jump part. This is a particular kind of a so-called \emph{Lévy-type operator} which allows us to derive the following theorem as special case of \cite[Thm. 4.2]{behme-invariant-2024}.

\begin{theorem} \label{distributional equation}
	Let $(\mathcal{A},\mathcal{D}(\mathcal{A}))$ be a Lévy-type operator of the form \eqref{Levy-type operator} in $(0,\infty)$ and assume that $C_c^\infty((0,\infty))\subseteq \mathcal{D}(\mathcal{A})$. Then an absolutely continuous probability measure with density $\nu$ is infinitesimally invariant for $\cA$ if and only if
	\begin{equation}
		\phi(x)\nu(x)+\overline{F}_\mu\ast\nu (x) = c, \quad x>0,\,c\in\mathbb{R}.
	\end{equation}
\end{theorem}

Hence, if we consider \eqref{levy langevin diffusion} with drift \eqref{drift}, we conclude that $\bm{\pi}(\dd x) = \pi(x)\dd x$ is an infinitesimally invariant measure for $(X_t)_{t\geq 0}$.

In order to define \emph{ergodicity}, we define for every positive and measurable function $f\geq 1$ the \emph{$f$-norm} of a signed measure $\nu$ on the measure space $((0,\infty),\mathcal{B}((0,\infty)))$ by 
\begin{equation*}
	\|\nu\|_f:= \sup_{g\leq f} \abs{\nu(g)}:=\sup_{g\leq f} \abs{\int g(\omega)\nu(\dd \omega)}.
\end{equation*}
The $1$-norm is also called \emph{total variation norm}. We say that a process $(X_t)_{t\geq 0}$ is \emph{$f$-ergodic} if an invariant probability measure $\nu$ for $(X_t)_{t\geq 0}$ exists and for all $x\in (0,\infty)$ it holds that
\begin{equation} \label{f-ergodicity}
	\lim_{t\to\infty}\|\PP^x(X_t\in\cdot)-\nu\|_f = 0.
\end{equation}
We say that a process is \emph{exponentially $f$-ergodic} if it is $f$-ergodic and the convergence in \eqref{f-ergodicity} is of exponential speed, i.e. there is $\beta \in (0,1)$ and $B\in (0,\infty)$ s.t.
\begin{equation*}
	\|\PP^x(X_t\in\cdot)-\nu\|_f \leq Bf(x)\beta^t.
\end{equation*}
Especially, we will claim $f$-ergodicity w.r.t. to a \emph{norm-like function} $f$, i.e. a function $f:[0,\infty)\to[0,\infty)$ such that $f(x)\to\infty$ as $x\to\infty$.

\subsection{Assumptions} \label{assumptions}

For the target measure $\bm{\pi}(\dd x)=\pi(x)\dd x$ we assume throughout that $\supp(\bm{\pi}) = [0,\infty)$ and $\bm{\pi}$ is absolutely continuous with strictly positive density almost surely. In particular, we assume that $\inf_{x\in C}\pi(x)>0$ for any compact $C\subset (0,\infty)$. Further, we assume that there exists $c>0$ s.t. $\int_0^x\pi(z)\dd z \leq cx\pi(x)$ for $x\ll 1$. As discussed in \cite[Sec. 3.1]{oechsler-levy-2024} this assumption ensures that the solution to \eqref{levy langevin diffusion} with drift \eqref{drift} does not drift onto $0$.

Additionally, we assume that $\pi$ is piecewise Lipschitz continuous. This means that for a partition $(x_i)_{i\in\mathbb{Z}}$ of $(0,\infty)$ with $x_i < x_{i+1}$ and unique accumulation point $0$ we have that the restriction $\pi\vert_{(x_i,x_{i+1})}$ is Lipschitz continuous with constant $K_i$ for all $i\in\mathbb{Z}$.

Lastly, as mentioned before, we assume that the driving Lévy process $L=(L_t)_{t\geq 0}$ is a compound Poisson process with jump distribution $\mu$ and intensity $1$. This implies the representation
\begin{equation} \label{representation compound poisson}
	L_t = \sum_{k=1}^{N_t} \xi_k
\end{equation}
with a Poisson process $N=(N_t)_{t\geq 0}$ with rate $1$ and an i.i.d. family of random variables $\{\xi_k, k\in\NN\}$ that is independent of $N$ with $\xi_k$ distributed according to $\mu$. We also assume that $\mu$ has a Lipschitz continuous distribution function $F_\mu$.

\subsection{Sampling from heavy-tailed target distributions}\label{Sec_heavy_tailed}

There is a multitude of concepts of heavy-tailedness in the literature all agreeing on the condition that such tails should be heavier than exponential, meaning a distribution $F$ for which 
\begin{equation*}
	\lim_{x\to\infty}e^{tx} \overline{F}(x) = \infty
\end{equation*}
for all $t>0$. Here we use the common notation $\overline{F}(x)=1-F(x)$ for the tail function. We write $F\in\mathcal{H}$ if $F$ belongs to this class of heavy-tailed distributions. Although the formulation of our main result does not require a particular notion of heavy-tailedness, we shortly introduce two important classes of heavy-tailed distributions to get acquainted with the concept. We begin with the class of subexponential distributions that is often considered to be the broadest class of heavy-tailed distributions relevant to applied mathematics. A distribution $F$ on $[0,\infty)$ with unbounded support is called \emph{subexponential} if
\begin{equation}\label{def subexponentiality}
	\overline{F^{\ast 2}}(x) \sim 2\overline{F}(x),
\end{equation}
where $F^{\ast 2}(x)=F\ast F(x)$ is a convolution power and $\sim$ denotes asymptotic equivalence, i.e. $\lim_{x\to\infty}\overline{F^{\ast 2}}(x)/2\overline{F}(x)=1$. We write $F\in\mathcal{S}$ if $F$ belongs to the class of subexponential distributions. Note that $\mathcal{S}\subset\mathcal{H}$, see \cite[Lemma 3.8]{nair-fundamentals-2022}. This definition becomes meaningful with the observation that, assuming two independent copies $X_1,X_2\sim F$, the right hand side of \eqref{def subexponentiality} fulfills in probabilistic notation
\begin{equation*}
	2\PP \left(X_1>x\right) \sim \PP\left(\max(X_1,X_2)>x\right),
\end{equation*}
see \cite[Eq. (3.4)]{nair-fundamentals-2022}. This property is often called the \emph{principle of the single big jump} since the probability of $X_1+X_2$ to exceed a large value is eventually determined by the probability of either $X_1$ or $X_2$ to exceed it.

The second class of heavy-tailed distributions we want to emphasise are regularly varying distributions. We say that a distribution $F$ on $[0,\infty)$ is \emph{regularly varying at infinity with index $-\rho$}, where $\rho \geq 0$, if for any $\lambda >0$ we have
\begin{equation*}
	\frac{\overline{F}(\lambda x)}{\overline{F}(x)} \to \lambda^{-\rho},
\end{equation*}
as $x\to\infty$. We write $F\in\mathcal{RV}(-\rho)$ and $\mathcal{RV}:=\cup_{\rho\geq 0}\mathcal{RV}(-\rho)$. Note that $\mathcal{RV}\subset \mathcal{S}$, see \cite[Lemma 2.18]{nair-fundamentals-2022}.

We write $f_1\in O(f_2)$ if $\limsup_{x\to\infty}\abs{f_1(x)}/f_2(x)<\infty$. We will also use $\lesssim$ as the notation for an asymptotic inequality up to a positive constant, i.e.
\begin{equation*}
	f_1(x)\lesssim f_2(x) \quad\Leftrightarrow\quad f_1(x)\leq c f_2(x) \quad\text{eventually for large $x$ and some $c>0$}.
\end{equation*}
We are now in the position to formulate the main result.

\begin{theorem} \label{sampling}
	Let the assumptions of Section \ref{assumptions} hold. Further, assume that the Lévy measure $\mu$ has unbounded support, and the distribution function $F_\mu$ is such that:
	\begin{enumerate}
		\item $\mu(I)>0$ for any open interval $I\subset (0,\infty)$.
		\item There exists a norm-like function $h$ with locally bounded derivative $h^\prime$ s.t. 
		\begin{equation} \label{integrability h}
			\sup_{x\in (0,\infty)}\int_0^\infty \left(h(x+z)-h(x)\right) \mu (\dd z) < \infty,
		\end{equation}
		and 
		\begin{equation} \label{tail and monotonicity}
			\frac{h(x)}{h^\prime (x)}\pi(x) \lesssim \overline{F}_\mu (x).
		\end{equation}
	\end{enumerate}
	Then, the strong solution of \eqref{levy langevin diffusion} with $\phi$ as in \eqref{drift} and $L$ with Lévy measure $\mu$ is exponentially $h$-ergodic towards $\bm{\pi}$.
\end{theorem}

\begin{remark} $\,$
	\begin{enumerate}
		\item In practice, the Lévy measure $\mu$ has the role of a canonical distribution one can easily sample from with sufficiently heavy tails. This is encoded in \eqref{tail and monotonicity}. This condition can be understood as a generalised monotonicity condition on the density $\pi$ intertwined with a tail domination condition on the distribution $\bm{\pi}$ w.r.t. $\mu$, see Example \ref{reg var} below.
		\item Various choices of $h$ as norm-like function fulfilling \eqref{integrability h} are possible depending on the concentration properties of $\mu$. Consider the following examples:
			\begin{enumerate}
				\item If $\mu$ has finite first moment, then $h(x):= x$ fulfills \eqref{integrability h}.
				\item If $\mu$ has finite $\rho$-th moment, $\rho>0$, then $h(x):=x^\rho$ fulfills \eqref{integrability h} due to its Hölder continuity.
				\item If $\mu$ has finite $h$-moment, i.e. $\int_0^\infty h\dd \mu <\infty$, and $h$ is eventually subadditive, then $h$ fulfills \eqref{integrability h}.
			\end{enumerate}
	\end{enumerate}
\end{remark}

\begin{example} \label{reg var}
	Let $\bm{\pi}\in\mathcal{RV}(-\rho)$ with eventually monotone density $\pi$ and choose $\mu=\text{Pareto}(x_m=0,\alpha=\rho^\prime)$ with $0<\rho^\prime<\rho$. Note, if we write $x_m=0$ we mean a Pareto distribution shifted to $0$. We will show that this setting fulfills the assumptions of Theorem \ref{sampling}: It is clear that Condition 1. is fulfilled by choice of $\mu$. It is commonly known from the theory of regular variation that a distribution that is regularly varying at infinity with index $-\rho$ has finite moments of order strictly smaller than $\rho$, c.f. \cite[Thm. 2.14]{nair-fundamentals-2022}. If $\rho>1$ we choose $\rho^\prime\in (1,\rho)$ and $h(x)=x$ and obtain \eqref{integrability h}. If $\rho\leq 1$ we set $h(x):=x^{\rho^{\prime\prime}}$ with $0<\rho^{\prime\prime}<\rho^\prime$ and obtain \eqref{integrability h} due to the Hölder continuity of $h$ and $\mu\in\mathcal{RV}(-\rho^\prime)$. Condition \eqref{tail and monotonicity} is true by choice of $\mu$ with $\rho^\prime<\rho$ and the fact that $\bm{\pi}\in\mathcal{RV}(-\rho)$ was chosen with eventually monotone density, which implies $x\pi(x)\sim c \overline{F}_{\bm{\pi}}(x)$, see \cite[Thm. 2.11]{nair-fundamentals-2022}. Here we use that for the choice of $h$ we have $h(x)/h^\prime(x)=\frac{1}{\rho^{\prime\prime}}x$ and $\overline{F}_\mu/\overline{F}_{\bm{\pi}}\in\mathcal{RV}(\rho-\rho^\prime),\rho-\rho^\prime >0$.
\end{example}

\begin{example}
	Some $\bm{\pi}\in\mathcal{H}\setminus\mathcal{RV}$ with tail decay faster than any polynomial can be chosen as target distribution: To begin, we make the observation that any such distribution must have finite first moment. To see this, we choose $\delta >0$ s.t. $\overline{F}_{\bm{\pi}}(x)\leq cx^{-2}$ for all $x\geq \delta$ and obtain via the tail sum formula
		\begin{equation*}
			\int_0^\infty x\bm{\pi}(\dd x) = \int_0^\infty \overline{F}_{\bm{\pi}}(x) \dd x \lesssim \delta + c\int_\delta^\infty x^{-2} \dd x <\infty.
		\end{equation*}
		We can therefore choose $h(x):=x$ in this case.
		\begin{enumerate}
			\item Weibull distribution: Assume that
			\begin{equation*}
				\pi(x)\sim \alpha\beta x^{\alpha-1}\exp\left(-\beta x^\alpha\right), \quad x,\beta\geq 0, \alpha\in (0,1),
			\end{equation*}
			and choose $\mu =\text{Weibull}(\alpha^\prime,\beta^\prime)$, hence $\overline{F}_\mu(x)=\exp(-\beta^\prime x^{\alpha^\prime})$. It is clear that the choice of $\mu$ as Weibull distribution with arbitrary parameters fulfills Condition 1. and \eqref{integrability h} of Theorem \ref{sampling}. Also, the choices $\alpha^\prime \in (0,\alpha)$ and any $\beta^\prime\geq 0$ imply that \eqref{tail and monotonicity} is true due to the fact that for any powers $0<\alpha^\prime<\alpha<1$ it holds that
			\begin{equation*}
				x^\alpha\exp\left(-\beta x^\alpha\right) \lesssim \exp\left(-\beta^\prime x^{\alpha^\prime}\right).
			\end{equation*}
			\item Lognormal distribution: Assume that 
			\begin{equation*}
				\pi(x) \sim \frac{1}{x\sigma_\pi\sqrt{2\pi}}\exp\left(-\frac{\left(\ln(x)-m_\pi\right)^2}{2\sigma_\pi^2}\right), \quad x,\sigma\geq 0, m\in\mathbb{R}.
			\end{equation*}
			and choose $\mu=\text{Lognormal}(m,\sigma^2)$, hence
			\begin{equation*}
				\overline{F}_\mu (x) = \int_x^\infty \frac{1}{u\sigma\sqrt{2\pi}}\exp\left(-\frac{\left(\ln(u)-m\right)^2}{2\sigma^2}\right) \dd u.
			\end{equation*}
			Condition 1. and \eqref{integrability h} of Theorem \ref{sampling} are clearly fulfilled in this case. Although the lognormal tail function is not accessible in closed form we can take advantage of \cite[Eq. 26.2.12]{abramowitz-handbook-1972} giving the infinite expansion for Gaussian integrals
			\begin{equation*}
				\int_x^\infty\exp\left(-\frac{u^2}{2}\right)\dd u = \frac{\exp\left(-\frac{x^2}{2}\right)}{x}\bigg(1-\frac{1}{x^2}+\frac{3}{x^4}+\ldots+\frac{(-1)^n3\cdots(2n-1)}{x^{2n}}\bigg) + R_n(x),
			\end{equation*}
			where $R_n$ has an integral representation and has absolute value less than the first neglected term of the series expansion. This allows to deduce via substitution $y = (\ln(u)-m)/\sigma$ in the estimate for $\overline{F}_\mu$ 
			\begin{equation*}
				\frac{\sqrt{2}\sigma}{\sqrt{\pi}(\ln(x)-m)}\exp\left(-\frac{(\ln(x)-m)^2}{2\sigma^2}\right)\leq \overline{F}_\mu (x),\quad x\in (0,\infty).
			\end{equation*}
			Thus, for the choice $\sigma^2>\sigma^2_\pi$ Assumption \eqref{tail and monotonicity} holds true since for the leading quadratic terms in the exponentials we have
			\begin{equation*}
				-\frac{1}{2\sigma_\pi^2}\ln(x)^2 \lesssim -\frac{1}{2\sigma^2}\ln(x)^2
			\end{equation*}
			and those determine the asymptotic comparison.
		\end{enumerate}
\end{example}

\begin{example}
	Some $\bm{\pi}\in\mathcal{H}\setminus\mathcal{RV}$ with tail decay slower than any polynomial can be chosen as target distribution: A common way to transform a distribution $\nu$ in order to obtain one with a heavier tail is to consider the random variable $Y=\exp(Z)$ where $Z\sim\nu$. Consider for example $Z\sim\text{Pareto}(x_m,\alpha),x_m> 0,\alpha>0$. Via substitution one sees that $Y$ has density
	\begin{equation*}
		f(x) = \frac{\alpha x_m^\alpha}{x\ln(x)^{\alpha+1}}, \quad x\geq \exp\left(x_m\right).
	\end{equation*}
	Assume that 
	\begin{equation*}
		\pi (x) \sim \frac{c}{x\ln(x)^{\alpha+1}}
	\end{equation*}
	for some $c>0$ and choose $\mu = \text{Logpareto}(x_m=0,\alpha^\prime)$ with $0<\alpha^\prime<\alpha$. By Logpareto we mean the distribution resulting from the above mentioned transformation, in this case of a shifted Pareto distribution. Then the choice $h(x):=\ln(x)^{\alpha^{\prime\prime}}$ with $0<\alpha^{\prime\prime}<\alpha^\prime$ yields a valid norm-like function due to the Lipschitz continuity of the logarithm and the Hölder continuity of $x\mapsto x^{\alpha^{\prime\prime}}$. In particular Condition 1. and \eqref{integrability h} of Theorem \ref{sampling} are fulfilled. Since 
	\begin{equation*}
		\frac{h(x)}{h^\prime (x)} = \frac{x\ln(x)^{\alpha^{\prime\prime}}}{\alpha^{\prime\prime}\ln(x)^{\alpha^{\prime\prime}-1}} = \frac{1}{\alpha^{\prime\prime}}x\ln (x)
	\end{equation*}
	it follows
	\begin{equation*}
		\frac{h(x)}{h^\prime (x)}\pi (x) \sim \frac{c}{\alpha^{\prime\prime}\ln (x)^\alpha}.
	\end{equation*}
	Finally, due to $\overline{F}_\mu (x) = \frac{x_m^{\alpha^\prime}}{\ln(x)^{\alpha^\prime}}$ one obtains \eqref{tail and monotonicity} for the choice $\alpha^\prime \leq \alpha$.
\end{example}

\subsection{Simulations}\label{Sec_simulate}

We illustrate Theorem \ref{sampling} on various examples that are obtained by implementing LLMC through the simulation of the jumps of $L$ and solving the autonomous equation \eqref{ode} in between jumps numerically. All examples are additionally plotted with a log-scale to emphasise that the samples are well distributed in the tail. We also compare all simulations to LLMC with a Lévy measure with exponential tails in order to show that convergence is significantly slower in that case. Further simulations, apart from the ones associated to these figures, revealed that even tripling the time of simulation LLMC with exponential Lévy measure still does not produce samples above $20$. This emphasises the difference in speed of convergence. We first consider the following two examples belonging to $\mathcal{S}\setminus\mathcal{RV}$. The density $f_1$ shows sine wave oscillations in $(0,10)$ and a Weibull tail. The density $f_2$ is inspired by Example $3.8$ from \cite{oechsler-levy-2024}, having discontinuities of jump-type and a lognormal tail.

\begin{example}[Oscillations and Weibull tail]
	Consider the density $f_1$ with normalising constant $c$ defined by
	\begin{align*}
		f_1:=
		\begin{cases}
			0,&\quad x \leq 0, \\
			c\exp\left(-\frac{1}{2}x\right),&\quad x\in (0,2.5], \\
			c\left(\frac{3}{2}+\sin\left(x^{\frac{3}{2}}\right)\right),&\quad x\in (2.5,10], \\
			cx^{-\frac{1}{2}}\exp\left(-x^{\frac{1}{2}}\right),&\quad x\in (10,\infty),
		\end{cases}
	\end{align*}
	with simulation results depicted in Figure 1.
\begin{figure}[ht]
	\center
	\includegraphics[scale=0.32]{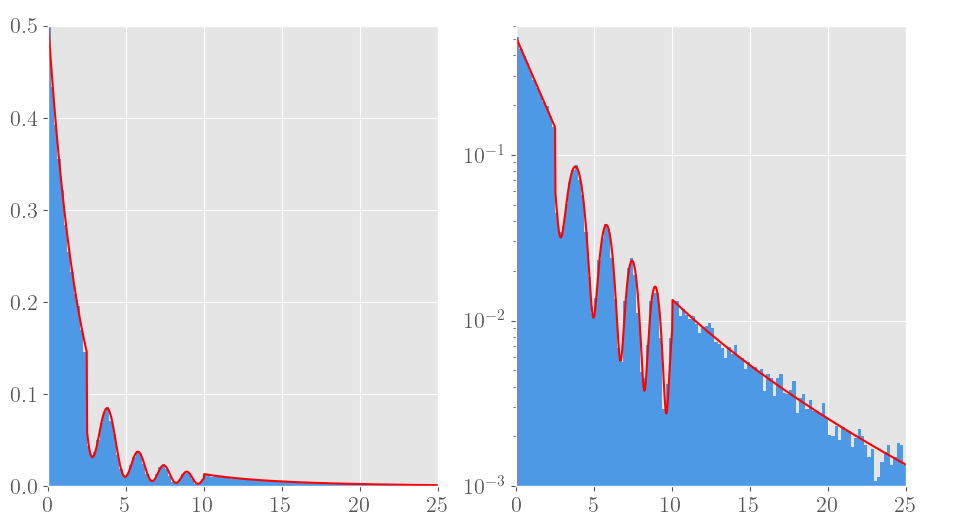}
	\includegraphics[scale=0.32]{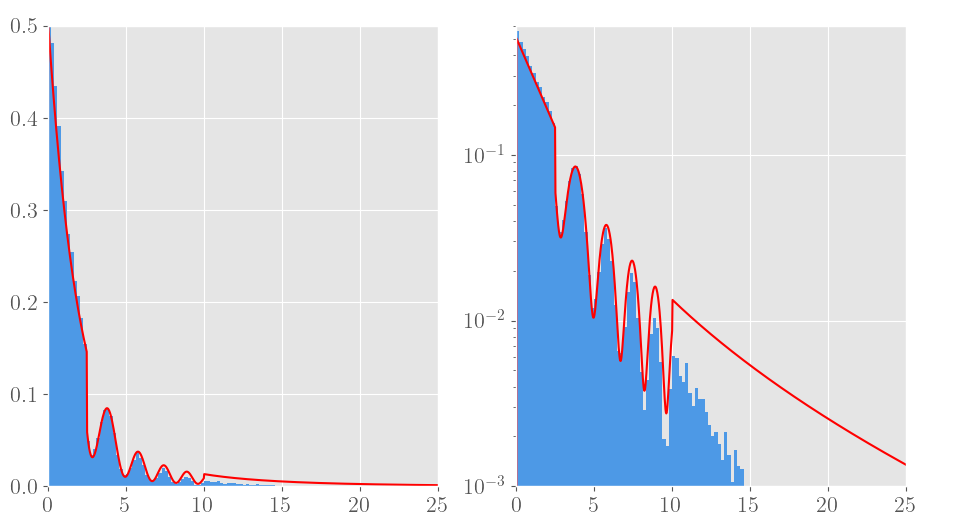}
	\caption{Density $f_1$ and histogram of $N=30.000$ samples after simulation for time $T=15$.\\
			Subfigure 1 and 2: Lévy noise with Weibull tail plotted with linear scale and log-scale. Subfigure 3 and 4: Lévy noise with exponential tail plotted with linear scale and log-scale.}
\end{figure}
\end{example}

\newpage
\begin{example}[Discontinuities and lognormal tail]
	Consider the density $f_2$ with normalising constant $c$ defined by
	\begin{align*}
		f_2:=
		\begin{cases}
			0,&\quad x \leq 0, \\
			c\exp(-\frac{1}{2}x),&\quad x\in (0,5], \\
			cx^{-2}+0.12,&\quad x\in (5,7], \\
			\frac{c}{x\sqrt{4\pi}} \exp\left(-\frac{\ln(x)^2}{4}\right),&\quad x\in (7,\infty),
		\end{cases}
	\end{align*}
	with simulation results depicted in Figure 2.
\begin{figure}[ht]
	\center
	\includegraphics[scale=0.32]{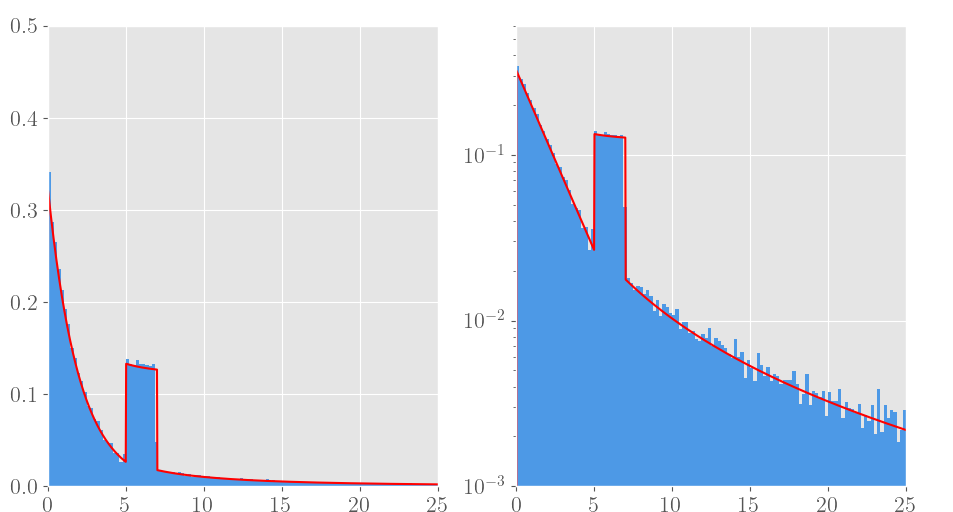}
	\includegraphics[scale=0.32]{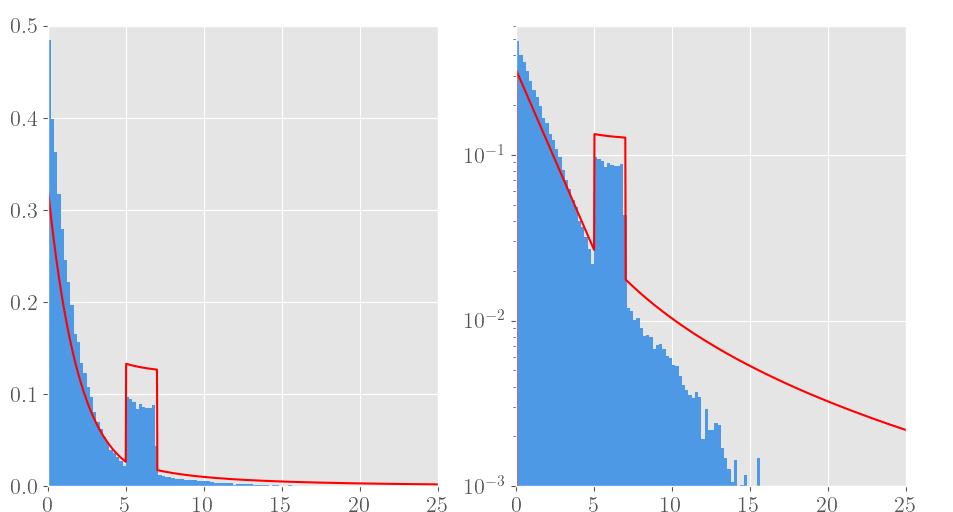}
	\caption{Density $f_2$ and histogram of $N=30.000$ samples after simulation for time $T=15$.\\
			Subfigure 1 and 2: Lévy noise with lognormal tail plotted with linear scale and log-scale. Subfigure 3 and 4: Lévy noise with exponential tail plotted with linear scale and log-scale.}
\end{figure}
\end{example}

\newpage Then, consider the following two examples from $\mathcal{RV}(-1)$. Again, $f_3$ is inspired by Example $3.8$ from \cite{oechsler-levy-2024}, having discontinuities of jump-type, here with a regularly varying tail. The density $f_4$ shows sine wave oscillations in $(0,10)$ and again a regularly varying tail. 

\begin{example}[Discontinuities and regularly varying tail] \label{example 3}
	Consider the density $f_3$ with normalising constant $c$ defined by
	\begin{align*}
		f_3 (x):=
		\begin{cases}
			0,&\quad x \leq 0, \\
			c\exp(-\frac{1}{2}x),&\quad x\in (0,5], \\
			cx^{-2}+0.12,&\quad x\in (5,7], \\
			cx^{-2},&\quad x\in (7,\infty),
		\end{cases}
	\end{align*}
	with simulation results depicted in Figure 3.
\begin{figure}[ht]
	\center
	\includegraphics[scale=0.32]{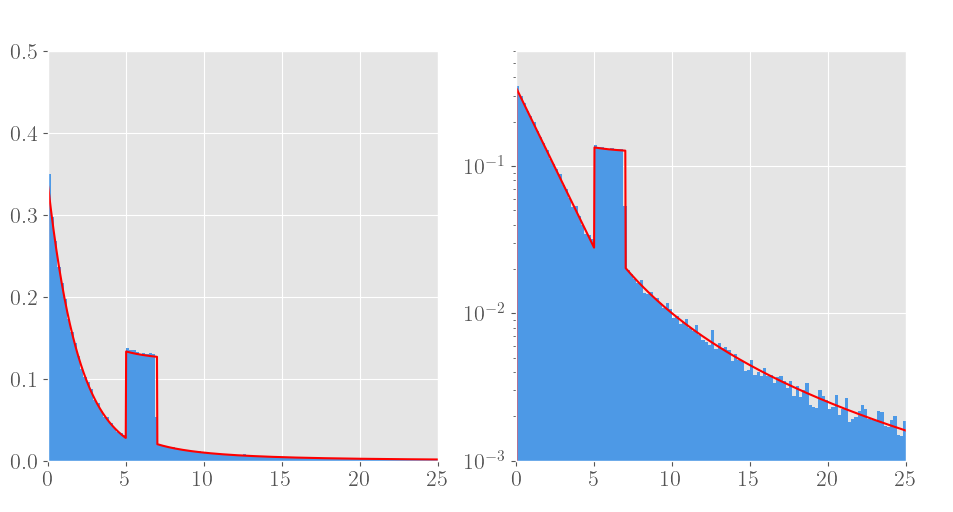}
	\includegraphics[scale=0.32]{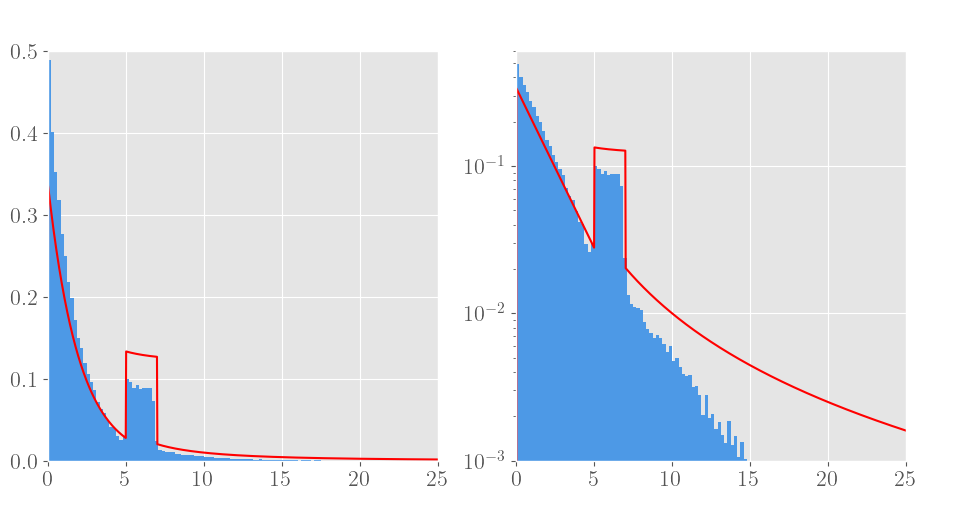}
	\caption{Density $f_3$ and histogram of $N=30.000$ samples after simulation for time $T=15$.\\
			Subfigure 1 and 2: Lévy noise with regularly varying tail plotted with linear scale and log-scale. Subfigure 3 and 4: Lévy noise with exponential tail plotted with linear scale and log-scale.}
\end{figure}
\end{example}
\newpage
\begin{example}[Oscillations and regularly varying tail] \label{example 4}
	Consider the density $f_4$ with normalising constant $c$ defined by
	\begin{align*}
		f_4(x):= 
		\begin{cases}
			0,&\quad x \leq 0, \\
			c\exp(-\frac{1}{2}x),&\quad x\in (0,2.5], \\
			c\left(\frac{3}{2}+\sin(x^{\frac{3}{2}})\right),&\quad x\in (2.5,10], \\
			cx^{-2},&\quad x\in (10,\infty),
		\end{cases}
	\end{align*}
	with simulation results depicted in Figure 4.
\begin{figure}[ht]
	\center
	\includegraphics[scale=0.32]{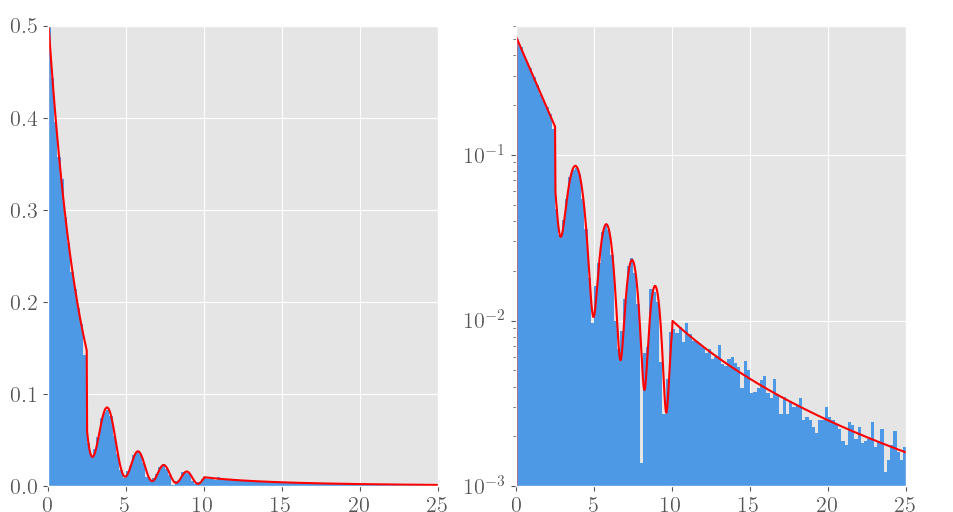}
	\includegraphics[scale=0.32]{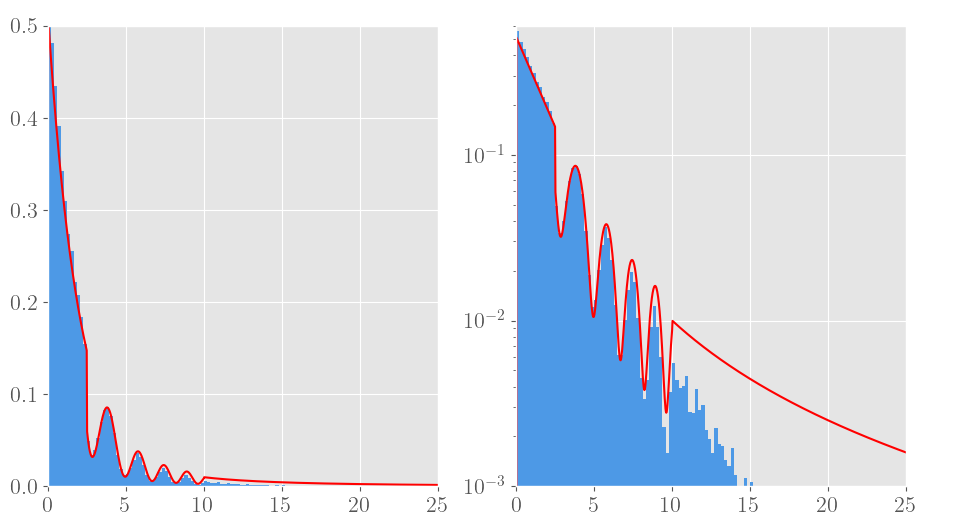}
	\caption{Density $f_4$ and histogram of $N=30.000$ samples after simulation for time $T=15$.\\
			Subfigure 1 and 2: Lévy noise with regularly varying tail plotted with linear scale and log-scale. Subfigure 3 and 4: Lévy noise with exponential tail plotted with linear scale and log-scale.}
\end{figure}
\end{example}

\section{Auxiliary results and proofs}\label{Sec_proofs}

In Section \ref{main} we have deduced the form of the drift coefficient $\phi$ that implies infinitesimal invariance of the target measure $\bm{\pi}$ for the strong solution of the SDE \eqref{levy langevin diffusion}. The strategy to deduce convergence of $(X_t)_{t\geq 0}$ to the target distribution is as follows. It is commonly known that a Markov process with a certain recurrence behaviour, which we are yet to define, has a unique invariant probability measure. In a series of three publications \cite{meyn-stability-1992,meyn-stability-1993,meyn-stability-1993-1}, Lyapunov criteria on the generator were developed in order to conclude this type of recurrence property. At this point, it is not clear whether this invariant probability measure is equal to the target measure, which is infinitesimally invariant by construction of $(X_t)_{t\geq 0}$. Similar to \cite{oechsler-levy-2024}, we are going to show that the distributional equation derived in \cite{behme-invariant-2024} has a unique solution and that, given the assumptions in Section \ref{assumptions}, any invariant measure is also infinitesimally invariant. Then we may conclude that the unique invariant probability measure is $\bm{\pi}$. The final step is to prove ergodicity, which essentially follows from showing the Lyapunov criteria for a particular type of norm-like function.

\subsection{Recurrence}

The type of recurrence we refer to in this section is  Harris recurrence, since it is known that Harris recurrence of a Markov process implies the existence of a unique invariant measure, cf. \cite[Thm. I.3]{azema-mesure-1967}. 
In order to define it, we first introduce the random variables associated to a Borel set $A$
\begin{equation*}
	\tau_A:=\inf\{t\geq 0\,\vert\, X_t\in A\}, \quad
	\eta_A:=\int_0^\infty\mathds{1}_A\left(X_t\right)\dd t.
\end{equation*}
Then a Markov process $X=(X_t)_{t\geq 0}$ on $(0,\infty)$ is \emph{Harris recurrent} if either
\begin{enumerate}
	\item for some $\sigma$-finite measure $\varphi$ it holds $\PP^x(\eta_A=\infty)=1$ if $\varphi(A)>0$, or
	\item for some $\sigma$-finite measure $\psi$ it holds $\PP^x(\tau_A<\infty)=1$ if $\psi(A)>0$.
\end{enumerate}
If the unique invariant measure of a Harris recurrent Markov process is finite, then the process is called \emph{positive Harris recurrent}.

We briefly comment on the constructions introduced in \cite{meyn-stability-1992,meyn-stability-1993,meyn-stability-1993-1} in order to make sense of the condition for positive Harris recurrence. So far, we have defined the pointwise infinitesimal generator, which is an extension of the strong infinitesimal generator, where the convergence in \eqref{pointwise generator} must hold uniformly. In aforementioned \cite{meyn-stability-1992,meyn-stability-1993,meyn-stability-1993-1}, the authors consider another extension of the strong generator, called the \emph{extended generator}, given by the pair $(\mathcal{G},\mathcal{D}(\mathcal{G}))$ with $\mathcal{G}f=g$ and characterised by the equations

\begin{equation*}
	\EE^x(f(X_t))=f(x)+\EE^x\bigg(\int_0^tg(X_s)\dd s\bigg)
\end{equation*}
and 
\begin{equation*}
	E^x\bigg(\int_0^t\vert g(X_s)\vert\dd s\bigg) <\infty,
\end{equation*}
where $\mathcal{D}(\mathcal{G})$ contains all measurable $f$ such that such a measurable function $g$ exists. Since the extended generator agrees with the pointwise infinitesimal generator for $(X_t)_{t\geq 0}$ on all functions relevant to this text we will stick to the latter notion.

Additionally, in \cite{meyn-stability-1993-1} the authors introduce a truncation of the process $X$, denoted by $X^m$ and defined by 
\begin{align*}
	X_t^m:=
		\begin{cases}
		X_t,\,\text{if}\quad t < \tau_m, \\
		\Delta_m,\,\text{if}\quad t\geq\tau_m,
		\end{cases}
\end{align*}
where $\tau_m:=\inf\{t\geq 0\,\vert\,X_t\in F_m^c\}$ is the first exit time of $F_m$ and $(F_m)_{m\in\mathbb{N}}$ is a family of open precompact sets with $F_m\uparrow (0,\infty)$, as $m\to \infty$. The truncation approach in \cite{meyn-stability-1993-1} circumvents difficulties of characterising the domain of the extended generator and the recurrence condition is then formulated w.r.t. $\mathcal{G}_m$ therein. For our purposes this construction is relevant since it implies that norm-like functions are in the domain of $\mathcal{G}_m$. In what follows we assume that this holds and notationally surpress the dependence on $m$, i.e. we just formulate our conditions w.r.t. $\mathcal{A}$ as given in Equation \eqref{pointwise generator}. 

Before, recall that a non-empty set $C\in\mathcal{B}((0,\infty))$ is called \emph{$\varphi$-petite} if there exists a probability measure $a$ on $(0,\infty)$ and a non-trivial measure $\varphi$ on $\mathcal{B}((0,\infty))$ s.t. for all $y\in C$
\begin{equation*}
\int_0^\infty \PP^y(X_t\in B) a(\dd t) \geq \varphi (B)
\end{equation*}
for all $B\in\mathcal{B}((0,\infty))$.

According to \cite[Thm. 4.2]{meyn-stability-1993-1}, the following condition on the pointwise generator allows to conclude that the associated Markov process is positive Harris recurrent.

\begin{condition}[Positive Harris recurrence] \label{recurrence condition}
There exist constants $c,d>0$, $g\geq 1$, a petite set $C$ and a norm-like function $h\geq 0$ that is bounded on $C$, such that
\begin{equation*}
\mathcal{A}h(x)\leq -cg(x) + d\mathds{1}_C(x), \quad x\in(0,\infty).
\end{equation*}
\end{condition}

The strategy to obtain the upper bound in Condition \ref{recurrence condition} for the generator in Equation \eqref{pointwise generator} is the following. Recall that for $(X_t)_{t\geq 0}$ solving \eqref{langevin diffusion} under the assumptions in Section \ref{assumptions} for $h\in C_c^\infty((0,\infty))$ the generator writes 
\begin{equation*}
	\mathcal{A}h(x) = \phi(x)h^\prime(x) + \int_0^\infty \big( h(x+z)-h(x) \big)\mu(\dd z).
\end{equation*}
We will choose a smooth norm-like function and obtain the same equation in that case. Note that this is well defined despite the fact that a norm-like function is not in $C_c^\infty((0,\infty))$ due to the aforementioned truncation procedure that we surpress in the notation. The goal will be to choose the norm-like function s.t. the integral term is bounded uniformly in $x$. This will be dealt with by the term $d\mathds{1}_C$ for a compact set $C$. It will then be the role of the drift term to achieve the bound $-cg(x)$ as $x\to\infty$. Note that our state space $(0,\infty)$ only needs asymptotic considerations towards $\infty$ and the behaviour of $\mathcal{A}f$ on compact sets of the form $[0,c]$ can be dealt with by $d\mathds{1}_C$ in Condition \ref{recurrence condition}. We will already formulate the proof for a particular choice of norm-like function that ensures $g=h$. The reason will become clear in Section \ref{ergodicity} below. But first, we show that all compact sets are petite.

\begin{lemma}[Compact sets are petite] \label{compacts petite}
Let the assumptions in Section \ref{assumptions} hold. Then  all compact sets are petite for $X=(X_t)_{t\geq 0}$ solving \eqref{langevin diffusion}.
\end{lemma}
\begin{proof}
This can be proven with a slight adaptation of the derivation in \cite[Lemma 4.2]{oechsler-levy-2024}. The strategy is to construct the measure $\varphi$ from all deterministic trajectories starting from a compact set $C$. W.l.o.g. we assume that $C=[k_1,k_2]$. Let $u_x$ be the solution to the autonomous differential equation
\begin{align} \label{ode on R}
\begin{cases}
u^\prime_x(t) = \phi\big(u_x(t)\big),\quad t\geq 0, \\
u_x(0)=x.
\end{cases}
\end{align}
Here, $^\prime$ denotes the differentiation w.r.t. the variable $t$ and the initial value is written in the subscript for clearer computations. Restricting on the event of no jumps up to time $t$, i.e. on $\{N_t=0\}$, allows for the bound
\begin{equation} \label{nojump}
\PP^x(X_t\in B) \geq \mathds{1}_B (u_x(t)) \PP(N_t=0) = \mathds{1}_B (u_x(t)) \exp(-t),
\end{equation}
for all $B\in\mathcal{B}((0,\infty))$.
Let $\delta\in (0,k_1)$. In the derivation of the measure $\varphi$ we need the quantity
\begin{equation*}
	D:=\sup_{x\in C}\sup_{t\in [u_x^{-1}(k_1),u_x^{-1}(k_1-\delta)]}\abs{u_x^\prime(t)} 
\end{equation*}
to be finite, where $u_x^{-1}(y)$ is the time that the solution starting in $x$ needs to arrive in $y$. First, we observe that the uniqueness and monotonicity of the solution to \eqref{ode on R} imply that $D$ can be realised by
\begin{equation*}
	D = \sup_{t\in [0,u_{k_1}^{-1}(k_1-\delta)]}\abs{u_{k_1}^\prime(t)}.
\end{equation*}
Second, by construction we have that $u_{k_1}([0,u_{k_1}^{-1}(k_1-\delta)])=[k_1-\delta,k_1]\subset(0,\infty)$. Further, $\phi$ is locally bounded on compact subsets of $(0,\infty)$. This is due to the local boundedness on compact subsets of $(0,\infty)$ of $\pi$ from below and of $\overline{F}_\mu\ast\pi\leq 1$ from above. Hence $D<\infty$ because of $u_{k_1}^\prime(t)=\phi(u_{k_1}(t))$. We set $a(\dd t):=\exp(-t)\dd t$ and obtain by restriction to the set $\{N_t=0\}$ on $(\Omega,\mathcal{F})$, with $(N_t)_{t\geq 0}$ as in \eqref{representation compound poisson}, and the $x$-dependent set $[u_x^{-1}(k_1),u_x^{-1}(k_1-\delta)]$ on the time axis for all $B\in\mathcal{B}((0,\infty))$ via \eqref{nojump}
\begin{align*}
\int_0^\infty \PP^x(X_t\in B) a(\dd t) 
&\geq \int_0^\infty \mathds{1}_B(u_x(t))\exp(-2t) \dd t \\
&\geq \int_{u_x^{-1}(k_1)}^{u_x^{-1}(k_1-\delta)} \mathds{1}_B(u_x(t))\exp(-2t) \dd t \\
&\geq \exp(-2u_{k_2}^{-1}(k_1-\delta)) \int_{u_x^{-1}(k_1)}^{u_x^{-1}(k_1-\delta)} \mathds{1}_B(u_x(t)) \dd t \\
&= \exp(-2u_{k_2}^{-1}(k_1-\delta)) \int_{u_x^{-1}(k_1)}^{u_x^{-1}(k_1-\delta)} \mathds{1}_B(u_x(t)) \frac{u^\prime_x(t)}{u^\prime_x(t)} \dd t \\
&\geq \frac{\exp(-2u_{k_2}^{-1}(k_1-\delta))}{D} \int_{u_x^{-1}(k_1)}^{u_x^{-1}(k_1-\delta)} \mathds{1}_B(u_x(t)) u^\prime_x(t) \dd t \\
&= \frac{\exp(-2u_{k_2}^{-1}(k_1-\delta))}{D} \int_{k_1-\delta}^{k_1} \mathds{1}_B(s) \dd s=:\varphi(B),
\end{align*}
a non-trivial measure due to finiteness of $D$.
\end{proof}

\begin{theorem}[Positive Harris recurrence]
	Let the assumptions of Section \ref{assumptions} hold. Further, assume that the Lévy measure $\mu$ has unbounded support and the distribution function $F_\mu$ fulfills Condition 2. of Theorem \ref{sampling}. Then, the strong solution of \eqref{levy langevin diffusion} with $\phi$ as in \eqref{drift} and $L$ with Lévy measure $\mu$ is positive Harris recurrent.
\end{theorem}
\begin{proof}
Let $h$ be as in Condition 2. of Theorem \ref{sampling}. Recall the form of the generator $\mathcal{A}$ in \eqref{Levy-type operator}, then for the numerator of the drift specified in \eqref{drift} term we observe that $\overline{F}_\mu(x)=1$ for all $x\leq 0$ and therefore
\begin{align*}
	\overline{F}_\mu\ast \pi(x) &= \int_0^x\overline{F}_\mu (x-u) \pi(u) \dd u \\
		&= \int_0^\infty\overline{F}_\mu (x-u) \pi(u) \dd u - \int_x^\infty \pi(u) \dd u \\
		&= \overline{F_\mu\ast F_{\bm{\pi}}}(x)-\overline{F}_{\bm{\pi}}(x).
\end{align*}
From this we obtain
\begin{align*}
	\mathcal{A}h(x) &= -h^\prime(x)\frac{\overline{F}_\mu\ast\pi(x)}{\pi(x)} + \int_0^\infty (h(x+z)-h(x)) \mu(\dd z) \\
	&\leq -h^\prime(x)\frac{\overline{F}_\mu\ast\pi(x)}{\pi(x)} + C \\
	&= - h(x)\frac{\overline{F_{\bm{\pi}}\ast F_\mu}(x)-\overline{F}_{\bm{\pi}}(x)}{\frac{h(x)}{h^\prime(x)}\pi(x)} + C,
\end{align*}
by Condition 2. \eqref{integrability h} of Theorem \ref{sampling}. Next, we decompose the denominator into
\begin{equation*}
	-\left(\overline{F_{\bm{\pi}}\ast F_\mu}(x)-\overline{F}_{\bm{\pi}}(x)\right)= -\left(\overline{F}_\mu(x)+\Delta(x)\right),
\end{equation*}
where the associated error is
\begin{align}
	\Delta(x) &:= \overline{F_{\bm{\pi}}\ast F_\mu}(x)-(\overline{F}_{\bm{\pi}}(x)+\overline{F}_\mu(x)) \nonumber \\
	&= \int_0^x \left(\overline{F}_{\bm{\pi}}(x-y)-\overline{F}_{\bm{\pi}}(x)\right)dF_\mu(y) - \overline{F}_{\bm{\pi}}(x)\overline{F}_\mu(x).\label{decomp}
\end{align}
The first term in \eqref{decomp} has positive sign due to $\overline{F}_{\bm{\pi}}(x-y)\geq\overline{F}_{\bm{\pi}}(x)$ for all $y\in [0,x]$ and can therefore be neglected. On the other hand, $\overline{F}_{\bm{\pi}}(x)\overline{F}_\mu(x)\geq 0$ and we obtain for all $x\in (0,\infty)$ 
\begin{equation*}
	-\left(\overline{F_{\bm{\pi}}\ast F_\mu}(x)-\overline{F}_{\bm{\pi}}(x)\right)\leq -\overline{F}_\mu(x)\left(1-\overline{F}_{\bm{\pi}}(x)\right)=-\overline{F}_\mu(x)F_{\bm{\pi}}(x),
\end{equation*}
which implies that
\begin{equation*}
	\mathcal{A}h(x) \leq -h(x)\frac{\overline{F}_\mu(x)F_{\bm{\pi}}(x)}{\frac{h(x)}{h^\prime(x)}\pi(x)} + C,
\end{equation*}
for all $x\in (0,\infty)$. Because of the Condition 2. \eqref{tail and monotonicity} of Theorem \ref{sampling} and $F_{\bm{\pi}}(x)\uparrow 1$ we finally know that there exists $c>0$ and $x_0\in(0,\infty)$ s.t. for all $x\geq x_0$
\begin{equation*}
	\mathcal{A}h(x) \leq - ch(x) + C.
\end{equation*}
Since $\lim_{x\to\infty}h(x)=\infty$ we know that we can choose $\tilde{x}_0\geq x_0$ and $\tilde{c}\in (0,c)$ to obtain
\begin{equation*}
	\mathcal{A}h(x) \leq - \tilde{c}h(x) + \tilde{C}\mathds{1}_{(0,\tilde{x}_0)}(x)
\end{equation*}
for all $x\in(0,\infty)$ with $\tilde{C}:=C+\max_{x\in (0,\tilde{x}_0)}\abs{\mathcal{A}h(x)}<\infty$ due to the local boundedness from above of $h^\prime$ and $\phi$. Indeed, for $x\ll 1$ we have $\phi(x)\leq cx$ due to the assumption $\int_0^x\pi(z)\dd z\leq cx\pi(x)$ and \eqref{drift}. To conclude, Condition \ref{recurrence condition} is fulfilled and positive Harris recurrence follows.
\end{proof}

\subsection{Unique invariant measure is target measure} \label{uniqueness}

As mentioned before, positive Harris recurrence implies the existence of a unique invariant probability distribution $\nu$ for $X$. We now show that this invariant measure is indeed the target measure we aimed for, i.e. we prove that $\nu=\bm{\pi}$. Before stating and proving the main result of this section, we need to introduce some notations and preparatory results concerning Schwarz distributions.

We denote by $L^1_{\mathrm{loc}}(\nu)$ the space of locally integrable functions w.r.t. $\nu$ and by $W^{k,p}_{\mathrm{loc}}(\nu)$ the local Sobolev space w.r.t. $\nu$, where $\nu$ is a measure on $(0,\infty)$.
For $\Omega\subseteq (0,\infty)$ open we define $\mathcal{D}(\Omega):=\mathcal{C}_c^\infty((0,\infty))$ and write \(\cD'(\Omega)\) for the space of all \emph{(Schwarz) distributions} \(T:\cD(\Omega) \to \RR, f\mapsto T(f)=\langle T, f\rangle \), i.e. the space of all linear functionals that are continuous w.r.t. uniform convergence on compact subsets of all derivatives.  \\
It is known that \(T\in\cD'(\Omega)\) if for all compact sets \(K\subset \Omega\) there exist a constant \(C_K>0\) and a number \(N_K\in\NN\) such that for all functions \(f\in \cC_c^\infty(\Omega)\) with \(\supp f\subset K\) it holds
\begin{equation}\label{def_distr}
	|\langle T, f\rangle|\leq C_K\max\{\|\partial^\alpha f\|_\infty: |\alpha|\leq N_K\}.
\end{equation}
This allows to define the \emph{order} \(N_T\in\NN\cup\{\infty\}\) of \(T\) as 
\begin{equation*}
	N_T:=\inf\{n\in\NN: \eqref{def_distr} \text{ holds for all } K\subset \Omega \text{ compact with } N_K=n\},
\end{equation*} 
and we write \(\cD'^k(\Omega)\) for the space of distributions of order $k$. Distributions of order $0$ can be represented by Radon measures: By the Riesz representation theorem, c.f. \cite[Thm. 4.11]{brezis-functional-2011}, for every distribution \(T\in\cD'^0(\Omega)\) there exists a real-valued Radon measure \(m\) on \((\Omega,\cB(\Omega))\) such that
\begin{equation}\label{eq_radon}
	\langle T, f\rangle = \int_{\Omega} f(x)m(\diff x)
\end{equation}
for all \(f\in \cC^0_c(\Omega)\). Vice versa, every real-valued Radon measure on \((\Omega,\cB(\Omega))\) defines a distribution in \(\cD'^0(\Omega)\) via the mapping \eqref{eq_radon}.\\
Further, a distribution \(T\in\cD'(\Omega)\) is called \emph{regular} if there exists a function $g\in L^1_{\text{loc}}(\dd x)$ such that
\begin{equation*}
	\langle T, f\rangle =\int_0^\infty f(x)g(x)\diff x
\end{equation*} 
for all \(f\in\cC^\infty_c(\Omega)\). From the above it is clear that all regular distributions are of order $0$. Since we can identify these types of distributions with their associated Radon-measure or locally integrable functions, we will sometimes abuse notation and use the same symbol for the two respective objects. We refer to \cite{brezis-functional-2011} and \cite{duistermaat-distributions-2010} for references of the above and further background information.

The following lemma collects useful statements about distributions. We provide this here without proof and only mention that the first item originates from \cite[Lem. 2.1]{behme-invariant-2024}, the second can be found in \cite[Thm. 3.1.4]{hoermander-analysis-2003} and the third is an immediate consequence of the definition of functions of bounded variation as distributional primitives of Radon measures, c.f. \cite[Def. 5.1.1]{ziemer-weakly-1989}.

\begin{lemma}[Regularity properties] \label{distributional regularity}
Let $\nu$ be a real-valued Radon measure on $(0,\infty)$.
\begin{enumerate}
\item If $g\in L^1_{\mathrm{loc}}(\nu)$, then $g\nu\in\mathcal{D}'^0 ((0,\infty))$.
\item If $\nu\in\mathcal{D}^\prime ((0,\infty))$ is s.t. $\langle\nu,f^\prime\rangle=0$ for all $f\in\mathcal{C}_c^\infty((0,\infty))$, i.e. $\nu$ has $0$ derivative, then $\nu$ is a constant in the sense that it is regular with associated function $g\equiv c\in\mathbb{R}$.
\item If $\nu\in\mathcal{D}'^0 ((0,\infty))$, then its distributional antiderivative is regular and unique up to a constant.
\end{enumerate}
\end{lemma}

We are now in the position to prove the main result of this section that allows to conclude, with help of Lemma \ref{approximation}, that the unique invariant measure of the Markov process $X$ is actually given by the target distribution $\bm{\pi}$.

\begin{theorem}[Invariant measures are infinitesimally invariant and unique] \label{uniqueness invariant measure}
Let the assumptions of Section \ref{assumptions} hold. Further, assume that $\mu$ is chosen s.t. $\inf\supp(\mu) >0$. Then any invariant measure for the strong solution $(X_t)_{t\geq 0}$ of \eqref{levy langevin diffusion} is infinitesimally invariant and $\bm{\pi}$ is the unique invariant measure.
\end{theorem}
\begin{proof}
The first part of the statement is a consequence of the assumption of local boundedness from below of $\pi$. Indeed, for $\varphi\in C_c^\infty((0,\infty))$, Itô's formula, cf. \cite[Thm. 20.7]{kallenberg-foundations-2021}, and the Lévy-Itô decomposition, see \cite[Thm. 16.2]{kallenberg-foundations-2021}, yield
\begin{equation}\label{eq-ito}
\varphi(X_t)-\varphi(x)=\int_0^t\phi(X_s)\varphi^\prime(X_s)\dd s + \int_0^t\int_0^\infty \big(\varphi(X_{s-}+z)-\varphi(X_{s-})\big)\eta(\dd s,\dd z)
\end{equation}
with the random measure $\eta(\dd s, \dd z):=\sum_t\delta_{(t,\Delta X_t)}(\dd s,\dd z)$ with compensator  $\EE(\eta(\dd s,\dd z))=\dd s\mu(\dd z)$. As for every $t\geq 0$ and $x\in (0,\infty)$
\begin{equation*}
\EE^x\bigg(\int_0^t \int_0^\infty \abs{\varphi(X_{s-}+z)-\varphi(X_{s-})}\mu(\dd z)\dd s\bigg) \leq 2t \| \varphi \|_\infty \mu((0,\infty)) <\infty,
\end{equation*}
and the replacement of $\eta$ with its compensator under the expectation is justified, cf. \cite[Thm. 2.21]{schnurr-symbol-2009}. Further we observe that $\phi \varphi^\prime$ is bounded, since $\varphi$ has compact support and $\phi$ is locally bounded, and therefore via Fubini's theorem and \eqref{eq-ito} we can estimate
\begin{align*}
\frac{1}{t}\big| \EE^x\left(\varphi(X_t)\right)-\varphi(x)\big|
&\leq \frac{1}{t}\bigg|\int_0^t \EE^x \left(\phi(X_s)\varphi(X_s)\right)\dd s\bigg| \\
&\quad + \frac{1}{t}\int_0^t\int_0^\infty \big| \EE^x \left(\varphi(X_{s-}+z)-\varphi(X_{s-})\right)\big| \mu(\dd z)\dd s\\ 
&\leq \|\phi \varphi\|_\infty+2 \| \varphi\|_{\infty} .
\end{align*}
We have found a uniform upper bound which is clearly integrable w.r.t. any probability measure. These considerations allow to apply Lebesgue's theorem to
\begin{equation*}
	\left(\frac{1}{t}\left( \EE^x\left(\varphi(X_t)\right)-\varphi(x)\right)\right)_{0\leq t\leq 1}
\end{equation*}
as $t\to 0$ and assuming that $\bm{\nu}$ is an invariant measure we obtain by exchange of limit and integration
\begin{equation*}
	0 = \lim_{t\to 0}\int_0^\infty\left(\frac{1}{t}\EE^x\left(\varphi(X_t)\right)-\varphi(x)\right) \bm{\nu}(\dd x) = \int_0^\infty \mathcal{A}\varphi(x) \bm{\nu} (\dd x),
\end{equation*}
hence $\bm{\nu}$ is infinitesimally invariant.

Concerning uniqueness, because of the first conclusion of this theorem, any invariant measure $\bm{\nu}$ has to fulfil the distributional equation
\begin{equation} \label{dist equation}
-(\phi\bm{\nu})^\prime - (\tilde{\mu}\ast\bm{\nu})^\prime = 0,
\end{equation}
by \cite[Thms. 4.2 or 4.3]{behme-invariant-2024}, where $^\prime$ denotes the distributional derivative and $\tilde{\mu}(\dd x):=\overline{F}_\mu(x)\dd x$ is the integrated tail of $\mu$. Note that in Equation \eqref{dist equation} $\phi\bm{\nu}$ is to be understood as the distribution associated to the measure $\phi(x)\bm{\nu}(\dd x)$ and the convolution is meant as a convolution of measures, i.e.
\begin{equation*}
\int_0^\infty \mathds{1}_B(x)(\tilde{\mu}\ast\bm{\nu})(\dd x) =\int_0^\infty\int_0^\infty \mathds{1}_B(x+u)\tilde{\mu}(\dd x)\,\bm{\nu}(\dd u)= \int_0^\infty\int_0^\infty \mathds{1}_B(x+u)\overline{F}_\mu(x)\dd x\,\bm{\nu}(\dd u),
\end{equation*}
for a Borel set $B$. In particular $\phi\bm{\nu}$ is a distribution due to the local boundedness of $\phi$ and the first statement of Lemma \ref{distributional regularity}. We can rewrite
\begin{align*}
	\langle (\tilde{\mu}\ast\bm{\nu})^\prime,\varphi\rangle
	&= \int_0^\infty\int_0^\infty \varphi^\prime (x+y)\overline{F}_\mu(y)\dd y\bm{\nu}(\dd x) \\[7pt]
	&= \int_0^\infty\int_x^\infty \varphi^\prime (y)\overline{F}_\mu(y-x)\dd y\bm{\nu}(\dd x) \\[7pt]
	&= -\int_0^\infty\int_x^\infty \varphi(y)\overline{F}_\mu^\prime(y-x)\dd y\bm{\nu}(\dd x) - \int_0^\infty \varphi(x)\bm{\nu} (\dd x) \\[7pt]
	&= -\int_0^\infty\varphi(y)\left(\int_0^y \overline{F}_\mu^\prime(y-x)\bm{\nu}(\dd x)\right)\dd y - \int_0^\infty \varphi(x)\bm{\nu} (\dd x) \\[7pt]
	&= \int_0^\infty \varphi(x) \tilde{\bm{\nu}}(\dd x)
\end{align*}
with 
\begin{equation*}
	\tilde{\bm{\nu}}(\dd x) = -\left(\int_0^x\overline{F}^\prime_\mu(x-u)\bm{\nu}(\dd u)\dd x + \bm{\nu}(\dd x)\right).
\end{equation*}
In the computations above we have shifted the integration variable in the second equality, applied integration by parts in the third, justified by the absolute continuity of $\mu$, and changed the order of integration via Fubini's theorem in the fourth equality. This means that $(\tilde{\mu}\ast\bm{\nu})^\prime\in\mathcal{D}^{\prime 0}((0,\infty))$ and therefore $(\phi\bm{\nu})^\prime = -(\tilde{\mu}\ast\bm{\nu})^\prime\in \mathcal{D}^{\prime 0}((0,\infty))$ as well. By the third statement of Lemma \ref{distributional regularity} we can conclude that $\phi\bm{\nu}$ is regular which is the case if and only if $\bm{\nu}$ is regular. Finally, as we are only considering probability measures, $\bm{\nu}$ can be identified with an element of $\nu\in L^1(\dd x)$. On the level of functions we now compute for $\varphi\in \mathcal{C}_c^\infty((0,\infty))$
\begin{align*}
-\langle (\phi\bm{\nu})^\prime, \varphi\rangle - \langle (\tilde{\mu}\ast\bm{\nu})^\prime,\varphi\rangle 
&= \int_0^\infty \varphi^\prime(x)\phi(x)\nu(x)\dd x + \int_0^\infty \int_0^\infty \varphi^\prime(x+y)\overline{F}_\mu(y)\nu(x)\dd y\,\dd x \\[7pt]
&= \int_0^\infty \varphi^\prime(x)\phi(x)\nu(x)\dd x + \int_0^\infty \varphi^\prime(x)\int_0^x \overline{F}_\mu(y)\nu(x-y)\dd y\,\dd x \\[7pt]
&= \int_0^\infty \varphi^\prime(x)\big(\phi(x)\nu(x)+\overline{F}_\mu\ast\nu(x)\big)\dd x = 0,
\end{align*}
where the convolution is now a convolution of functions. Therefore, the second statement of Lemma \ref{distributional regularity} implies that
\begin{equation} \label{ode distributional solution}
\phi(x)\nu(x)+\overline{F}_\mu\ast\nu(x) = c \in\mathbb{R}.
\end{equation}
We claim that $c=0$. First of all, for any $x\in (0,\inf\supp(\mu))$ Equation \eqref{ode distributional solution} simplifies to the ODE in the generalised sense of Carathéodory
\begin{equation} \label{ode below support}
-\phi(x)\nu(x)=\int_0^x\nu(u)\dd u - c.
\end{equation}
This simplification stems from the fact that $\overline{F}_\mu(x)=1$ for all $x\in(0,\inf\supp(\mu))$. Since $\nu$ is a probability density and $-\phi> 0$ on $(0,\inf\supp(\mu))$
\begin{equation*}
\int_0^x\nu(u)\dd u  \geq c
\end{equation*}
for all $x\in (0,\inf\supp(\mu))$ and therefore $c\leq 0$. If we assume that $c<0$ then for any $x\in (0,\inf\supp(\mu))$ by lower bounding the right hand side of \eqref{ode below support} with $-c$, dividing by $-\phi$ and integrating we obtain
\begin{equation} \label{no drift onto 0}
\int_0^x\nu(u)\dd u \geq c \int_0^x\frac{1}{\phi(u)} \dd u = \infty,
\end{equation}
which contradicts the necessary integrability of $\nu$. Indeed, the divergence of the integral is due to the assumption $\int_0^x\pi(u)\dd u \leq c x\pi(x)$ for $x\ll 1$. We conclude that $c=0$.

From this point onwards the argumentation is analogous to \cite[Thm. 3.4 (iii)]{oechsler-levy-2024}. In more detail, one obtains a unique solution to \eqref{ode distributional solution} on $x\in(0,\inf\supp(\mu))$ by solving the ODE \eqref{ode below support} via Carathéodory's existence theorem, cf. \cite[Thm. 5.1]{hale-ordinary-1969}. The unique solution, up to a normalising constant, is then given by $\nu=\pi$. The fact that $\inf\supp(\mu)>0$ then implies that for $x\in[\inf\supp(\mu),2\inf\supp(\mu))$ it holds
\begin{equation*}
\int_0^xF_\mu(u)\nu(x-u)\dd u = \int^x_{\inf \mathrm{supp}(\mu)}F_\mu(u)\nu(x-u)\dd u = \int_0^xF_\mu(u)\pi(x-u)\dd u.
\end{equation*}
which allows for subsequent application of Carathéodory's existence theorem to obtain a unique solution on $[\inf\supp(\mu),2\inf\supp(\mu))$. Repeating this argument yields the unique solution $\nu=\pi$. Thus the invariant probability measure must be $\bm{\pi}$.
\end{proof}

Theorem \ref{sampling} relies on the assumption that $\mu(I)>0$ for all open intervals $I\subseteq (0,\infty)$. This assumption is clearly violated in the formulation of Theorem \ref{uniqueness invariant measure}. In order to take advantage of Theorem \ref{uniqueness invariant measure}, we construct a sequence $(X_t^n)_{t\geq 0}$ of processes that fulfill its assumptions and at the same time converges almost surely on compact sets w.r.t. $t$ to a solution $(X_t)_{t\geq 0}$ of \eqref{levy langevin diffusion} under the assumptions of Theorem \ref{sampling}. We start our exposure with two results that elaborate how weak convergence preserves invariant measures, and how Lipschitz properties of the target density translate to $\phi$, before constructing the mentioned sequence.

\begin{lemma}[Weak convergence preserves invariant measure]\label{lemweakconvergence}
Let $(X_t)_{t\geq 0}$ and $(X_t^n)_{t\geq 0}$ for each $n\in\mathbb{N}$ be a Markov process, all with state space $(0,\infty)$ and defined on the same probability space. Assume that $X_t^n\to X_t$ weakly as $n\to \infty$ for all $t\geq 0$, and that $\eta$ is an invariant measure for $(X_t^n)_{t\geq 0}$ for all $n\in\mathbb{N}$. Then $\eta$ is an invariant measure for $(X_t)_{t\geq 0}$ as well.
\end{lemma}
\begin{proof}
It is known, since $(0,\infty)$ is locally compact, that $\mathcal{C}_c((0,\infty))\cap\text{Lip}_1((0,\infty),[0,1])$ is a separating class, where $\text{Lip}_1((0,\infty),[0,1])$ are the Lipschitz continuous functions mapping from $(0,\infty)$ to $[0,1]$ with Lipschitz constant $1$, cf. \cite[Thm. 13.11]{klenke-probability-2013}. This means that if for two measures $\nu_1$ and $\nu_2$ and all $\varphi\in \mathcal{C}_c((0,\infty))\cap\text{Lip}_1((0,\infty),[0,1])$ it holds
\begin{equation*}
\int \varphi(x) \nu_1(\dd x) = \int \varphi(x) \nu_2(\dd x),
\end{equation*}
then $\nu_1=\nu_2$.
It is also known that the space of bounded and continuous functions, $\mathcal{C}_b((0,\infty))$, is convergence determining, meaning that $X_t^n\to X_t$ weakly is equivalent to $\EE^x(\varphi(X_t^n))\to \EE^x(\varphi(X_t))$ for all $\varphi\in\mathcal{C}_b((0,\infty))$ and $x\in (0,\infty)$, cf. \cite[Thm. 13.16]{klenke-probability-2013}. The assertion is then a consequence of $\mathcal{C}_c((0,\infty))\cap\text{Lip}_1((0,\infty),[0,1])\subseteq \mathcal{C}_b((0,\infty))$. Indeed, let $\eta$ be invariant for $(X_t^n)_{t\geq 0}$ for all $n\in\mathbb{N}$ and define
\begin{equation*}
\tilde{\eta}_t(\dd u) := \int_0^\infty \PP^x(X_t \in \dd u) \eta(\dd x).
\end{equation*}
We have to show that $\tilde{\eta}_t=\eta$ for all $t\geq 0$. We pick $\varphi\in \mathcal{C}_c((0,\infty))\cap\text{Lip}_1((0,\infty),[0,1])$ and an arbitrary $t\geq0$ and compute
\begin{align*}
\int_0^\infty \varphi(u)\tilde{\eta}_t(\dd u) 
&= \int_0^\infty \EE^x(\varphi(X_t))\eta(\dd x) \\[7pt]
&= \int_0^\infty \lim_{n\to\infty} \EE^x(\varphi(X_t^n))\eta(\dd x) \\[7pt]
&= \lim_{n\to\infty}\int_0^\infty \EE^x(\varphi(X_t^n))\eta(\dd x) \\[7pt]
&= \int_0^\infty \varphi(u)\eta(\dd u).
\end{align*}
The exchange of limits is justified by Lebegue's theorem since $\EE^x(\varphi(X_t^n))$ is uniformly bounded in $x$ by the boundedness of $\varphi$. The last equality sign is justified by the assumption that $\eta$ is invariant for $(X_t^n)_{t\geq 0}$ for all $n\in\mathbb{N}$. Hence, $\tilde{\eta}_t=\eta$ for all $t\geq 0$ which proves the claim.
\end{proof}

\begin{lemma}[Regularity of $\phi$] \label{Lipschitz continuity coefficient}
Let  $\pi$ be the density of the target distribution $\bm{\pi}$, obeying the assumptions in Section \ref{assumptions}. Let $g:[0,\infty)\to[0,\infty)$ be bounded and Lipschitz continuous on $[0,T]$ for any $T>0$. Then the following statements hold.
\begin{enumerate}
\item Then $g\ast\pi$ is piecewise Lipschitz continuous w.r.t. the same partition as $\pi$.
\item Set $\phi(x):=g\ast\pi(x)/\pi(x)$. Then $\phi$ is piecewise Lipschitz continuous w.r.t. the same partition as $\pi$.
\end{enumerate}
\end{lemma}
	\begin{proof}$\,$
Note that according to Section \ref{assumptions}, $\pi$ is piecewise Lipschitz continuous w.r.t. a partition $(x_i)_{i\in\mathbb{Z}}$ of $(0,\infty)$. \\
$1.$ Let $x,y\in (x_i,x_{i+1})$ with $x<y$, then introducing the term $\int_0^yg(x-u)\pi(u)\dd u$ yields 
\begin{align*}
\abs{g\ast\pi(y)-g\ast\pi(x)}
&= \bigg| \int_0^y g(y-u) \pi(u) \dd u - \int_0^x g(x-u)\pi(u) \dd u\bigg| \\
&\leq \int_0^y\abs{g(y-u)-g(x-u)}\pi(u)\dd u + \int_x^yg(x-u)\pi(u)\dd u \\
&\leq C \abs{y-x} + \sup_{x\in (x_i,x_{i+1})}\pi(x)\int_0^{y-x}g(u)\dd u \\
&\leq \left(C+C_{\pi,i}\cdot\tilde{C}\right)\abs{y-x},
\end{align*}
with $C_{\pi,i}:= \sup_{[x_i,x_{i+1}]}\pi(x)$. The second inequality is a direct consequence of the properties of $g$ and $\pi$, namely the Lipschitz continuity of $g$ with Lipschitz constant $C$, and the boundedness of $\pi$ on $(x_i,x_{i+1})$ since it is Lipschitz continuous on that set. Further, the mapping $G:x\mapsto\int_0^xg(u)\dd u$ is linearly bounded, since it is absolutely continuous, has bounded derivative and fulfills $G(0)=0$, justifying the last inequality, where $\tilde{C}$ is the coefficient associated to the linear boundedness of $G$.\\
$2.$ Again choose $x,y\in (x_i,x_{i+1})$ with $x<y$, then
\begin{align*}
\abs{\phi(x)-\phi(y)} &= \abs{\frac{\pi(y) \left(g \ast\pi\right)(x)-\pi(x)\left(g\ast\pi\right)(y)}{\pi(x)\pi(y)}} \\
&\leq c_{\pi,i}^{-2}\Big(\big| \pi(y)\big(g\ast\pi(x)-g\ast\pi(y)\big)-\big(\pi(x)-\pi(y)\big)g\ast\pi(y)\big| \Big)\\
&\leq \frac{C_{\pi,i}}{c_{\pi,i}^2}\abs{g\ast\pi(x)-g\ast\pi(y)} +\frac{C_{\ast,i}}{c_{\pi,i}^2}\abs{\pi(x)-\pi(y)},
\end{align*}
with constants
\begin{equation*}
	c_{\pi,i}:=\inf_{x\in[x_i,x_{i+1}]} \pi(x) \quad\text{and}\quad C_{\ast,i}:=\sup_{u\in [x_i,x_{i+1}]}g\ast\pi(u) \quad\text{and}\quad \text{$C_{\pi,i}$ as above}.
\end{equation*}
Further, by $1.$, $g\ast\pi$ is Lipschitz continuous on $(x_i,x_{i+1})$ and we write $\tilde{c}$ for the associated Lipschitz constant, while  $c$ is the Lipschitz constant for $\pi$ on $(x_i,x_{i+1})$. Then,
\begin{equation*} 
\abs{\phi(x)-\phi(y)} \leq \frac{\tilde{c}\cdot C_{\pi,i}+c\cdot C_{\ast,i}}{c_{\pi,i}^2}\abs{x-y} ,
\end{equation*}
which finishes the proof.
\end{proof}

\begin{theorem}[Approximation of $(X_t)_{t\geq 0}$] \label{approximation}
Let the assumptions of Section \ref{assumptions} hold. Further, assume that $\mu$ fulfills Condition 1. of Theorem \ref{sampling}. Let $L$ have representation \eqref{representation compound poisson} and define the compound Poisson processes $L^n=(L_t^n)_{t\geq 0}$, $n\in \NN$, via
\begin{equation*}
	L_t^n := \sum_{k=1}^{N_t} \bigg(\xi_k \mathds{1}_{\xi_k>\frac1n} + \frac{1}{n} \mathds{1}_{\xi_k\leq \frac1n}\bigg), \quad t\geq 0,	
\end{equation*}
then $L^n$ has Lévy measure $\mu_n$ given by
\begin{equation*}
	\mu_n(B)=\mu\Big(\big[0,\tfrac{1}{n}\big]\Big)\delta_{\frac{1}{n}}(B) + \mu\Big(\big(\tfrac{1}{n},\infty\big)\cap B\Big), \quad B\in \cB((0,\infty)).
\end{equation*} 
Let $(X_t)_{t\geq 0}$ be the strong solution of \eqref{levy langevin diffusion} and for every $n\in\NN$ let $(X_t^n)_{t\geq 0}$ be the strong solution of
\begin{equation*}
\dd X_t^n = \phi_n(X_t^n)\dd t + \dd L_t^n, \quad t\geq 0,
\end{equation*}
with $\phi_n(x):=-\overline{F}_{\mu_n}\ast\pi(x)/\pi(x)$ and such that $\PP^x(X_0=x)=\PP^x(X_0^n=x)=1$. Then 
\begin{equation*}
\lim_{n\to\infty} \sup_{t\in[0,T]}\vert X_t^n-X_t\vert =0 
\end{equation*}
almost surely w.r.t. $\PP^x$ for all $T>0,x\in (0,\infty)$.
\end{theorem}
\begin{proof}
Let us fix $T,x>0$ and a sample $\omega\in\Omega$. By definition, we have $X_0=x=X_0^n$ almost surely w.r.t. $\PP^x$ and $(X_t^n)_{t\geq 0}$ fulfills the conditions of Theorem \ref{uniqueness invariant measure} for every $n\in\mathbb{N}$, implying that $\bm{\pi}$ is its unique invariant measure. We consider the paths $\chi(t):=X_t(\omega)$ and $\chi_n(t):=X_t^n(\omega)$.
To begin with, observe  that $\overline{F}_\mu=\overline{F}_{\mu_n}$ on $[\frac{1}{n},\infty)$, while for $x\in [0,\frac{1}{n})$
\begin{equation*}
\overline{F}_{\mu_n} (x)-\overline{F}_\mu(x) = \mu((0,x)) \leq\mu\big((0,\tfrac{1}{n})\big).
\end{equation*}
Thus, for $x\in(0,\infty)$
\begin{align} \label{phi_phin}
\abs{\phi(x)-\phi_n(x)} 
&= \frac{1}{\pi(x)}\bigg| \int_0^x (\overline{F}_{\mu_n}-\overline{F}_\mu)(u)\pi(x-u)\dd u \bigg|\leq \mu\left(\left(0,\tfrac{1}{n}\right)\right)\int_0^{\frac{1}{n}\wedge x} \frac{\pi(x-u)}{\pi(x)}\dd u.
\end{align}
By assumption, $\pi$ is a probability density and locally bounded away from $0$. Thus, for any compact set $K\subseteq (0,\infty)$ we have $\sup_{x\in K}\int_0^{\frac{1}{n}\wedge x} \frac{\pi(x-u)}{\pi(x)}\dd u < c(K)^{-1} <\infty$ for some positive constant $c(K)$. Therefore, locally $\abs{\phi(x)-\phi_n(x)}\to 0$ as $n\to \infty$ due to $\mu ((0,1/n))\to 0$.\\
Recall that $(x_i)_{i\in\mathbb{Z}}$ is the partition w.r.t. which $\pi$ is piecewise Lipschitz continuous and note that $x\in(x_i,x_{i+1}]$ for some $i\in\mathbb{Z}$. We define the time of the first jump or arrival at some $\pi$-discontinuity of $\chi(t)$ via
\begin{equation*}
t_1:=\inf\{t\geq 0\,\vert\, \Delta \chi(t)\neq 0\,\,\text{or}\,\,\chi(t)\in (x_i)_{i\in\mathbb{Z}}\},
\end{equation*}
and analogously
\begin{equation*}
t_1(n):=\inf\{t\geq 0\,\vert\, \Delta \chi_n(t)\neq 0\,\,\text{or}\,\,\chi_n(t)\in (x_i)_{i\in\mathbb{Z}}\}.
\end{equation*}
The aim is to analyse the behaviour of the paths $\chi(t)$ and $\chi_n(t)$ on $[0,t_1)$ and $[0,t_1(n))$ respectively. W.l.o.g. we assume that $x\in (x_i,x_{i+1})$ otherwise define $t_1$ and $t_1(n)$ with $(x_j)_{j\in\mathbb{Z}}\setminus\{x_{i+1}\}$. We discuss the two cases separately, a jump or arrival at a discontinuity. Assume first that there is a jump at $t_1$ and note that by construction the jump times of $(X_t)_{t\geq 0}$ and $(X_t^n)_{t\geq 0}$ coincide. We also observe that by construction $\phi_n(x)<\phi(x)<0$ for all $x>0$, and because of this on $(0,t_1)$ the path $\chi_n$ decays faster than $\chi$ for all $n\in\mathbb{N}$. We claim that $t_1(n)\uparrow t_1$ and assume for contradiction that $t_1(n)\uparrow\tilde{t}_1<t_1$. Since the jump times conincide it must hold that $\chi_n(t_1(n))=x_i$. Let us fix $m\in\mathbb{N}$ and conclude that for all $n\geq m$ and $t\in [0,t_1(m))$,
\begin{align*}
\abs{\chi(t)-\chi_n(t)} &\leq \int_0^{t_1(m)}\left(\abs{\phi(\chi(s))-\phi(\chi_n(s))} + \abs{\phi(\chi_n(s))-\phi_n(\chi_n(s))}\right) \dd s \\ 
&\leq \tilde{K}_i\int_0^{t_1(m)}\abs{\chi(s)-\chi_n(s)} \dd s + t_1(m)\frac{\mu\left((0,\frac{1}{n})\right)}{c_{\pi,i}},
\end{align*}
since the paths obey the ODE \eqref{ode} with $\phi$ or $\phi_n$ respectively. Here, we have used the piecewise Lipschitz continuity of $\phi$ with Lipschitz constant $\tilde{K}_i$ associated to $(x_i,x_{i+1})$, justified by Lemma \ref{Lipschitz continuity coefficient}, as well as the bound \eqref{phi_phin} and local lower bound of $\pi$ associated to $(x_i,x_{i+1})$ denoted by $c_{\pi,i}$. Applying Gronwall's lemma, cf. \cite[Lem. 26.9]{klenke-probability-2013}, then yields
\begin{equation} \label{path convergence}
\abs{\chi(t)-\chi_n(t)} \to 0
\end{equation}
as $n\to\infty$ for all $t\in [0,t_1(m))$. As $t_1(n)\uparrow\tilde{t}_1$ this extends to $[0,\tilde{t}_1)$, and since $\tilde{t}_1<t_1$ and $t_1$ is the first jump, we conclude via the continuity of the paths that this even extends to $[0,\tilde{t}_1]$. Similarly, for any $t\in[0,t_1)$ we have that $\chi(t)-x_i\geq\varepsilon$ for some $\varepsilon >0$. This implies that there exists some $n_0$ such that $\abs{\chi(\tilde{t_1})-\chi_n(\tilde{t}_1)}\geq\varepsilon$ for all $n\geq n_0$. This is a contradiction and therefore $t_1(n)\uparrow t_1$ implying that \eqref{path convergence} holds on $[0,t_1]$, in particular because $\chi$ and $\chi_n$ only differ in jumps of size smaller than $1/n$ by at most $1/n$. \\ 
At this point, similarly to \cite[Thm. 3.4 (iii)]{oechsler-levy-2024}, one can iterate this step. To formalise this we define
\begin{equation*}
	t_2:=\inf\{t\geq t_1\,\vert\, \Delta \chi(t)\neq 0\,\text{or}\,\chi(t)\in (x_i)_{i\in\mathbb{Z}}\}
\end{equation*}
and $t_2(n)$ analogously. Still in the case of $t_1$ being associated to a jump, there exists $n_0\in\mathbb{N}$ s.t. $\Delta L_{t_1}(\omega)=\Delta L_{t_1}^n(\omega)$ for all $n\geq n_0$, due to the coupling of $L$ and $L^n$. In conclusion, for any $\varepsilon>0$ there exists $n_0\in\mathbb{N}$ s.t. for all $n\geq n_0$ we have $\abs{\chi(t_1)-\chi_n(t_1)}\leq\varepsilon$. Choosing $\varepsilon$ small enough ensures that $\chi$ and $\chi_n$ jump into the same interval $(x_\ell,x_{\ell+1})$ where $\phi$ is Lipschitz continuous. It is important to note that by the jump properties of the underlying Poisson process $N_t$ and the discontinuities $(x_i)_{i\in\mathbb{Z}}$, the probability of a jump and a discontinuity at the same time is $0$ for $(X_t)_{t\geq 0}$ and $(X_t^n)_{t\geq 0}$ for all $n\in\mathbb{N}$. Also the occurence of a jump into a discontinuity has probability $0$. This is evident for $(X_t)_{t\geq 0}$, but also for $(X_t^n)_{t\geq 0}$ it is true since we only have one Dirac measure at $\frac{1}{n}$, which implies that for this to happen we must jump exactly at one time point. This means that eventually $t_2(n)>t_1$ and we can repeat the previous reasoning on $[t_1,t_2)$ by estimating for all $t\in[t_1,t_2)$
\begin{equation} \label{iteration}
	\abs{\chi(t)-\chi_n(t)} \leq \abs{\chi(t_1)-\chi_n(t_1)} + \int_{t_1}^{t_2} \abs{\phi(\chi(s))-\phi(\chi_n(s))} \dd s + \int_{t_1}^{t_2} \abs{\phi(\chi_n(s))-\phi_n(\chi_n(s))} \dd s.
\end{equation}
The first and third summands converge to $0$ as $n\to\infty$ due to our previous arguments. The second summand can be decomposed in the following way
\begin{equation*}
	\int_{t_1}^{t_2} \abs{\phi(\chi(s))-\phi(\chi_n(s))} \dd s = \int_{t_1}^{t_2(n)} \abs{\phi(\chi(s))-\phi(\chi_n(s))} \dd s + \int_{t_2(n)}^{t_2} \abs{\phi(\chi(s))-\phi(\chi_n(s))} \dd s.
\end{equation*}
The second term converges to $0$ as $n\to\infty$ due to local boundedness of $\phi$ and the first permits using the Lipschitz continuity on $(x_\ell,x_{\ell +1})$ and subsequently a repeated application of Gronwall's lemma.

Now assume the other case, meaning that $\chi$ hits a discontinuity of $\phi$ in $t_1$. The same application of Gronwall's lemma yields \eqref{path convergence} on sets of the form $[0,t_1(m))$. Again, we claim that $t_1(n)\uparrow t_1$ in this case and assume for contradiction that $t_1(n)\uparrow\tilde{t}_1<t_1$. As before we deduce that $\chi_n(t)\to \chi(t)$ for all $t\in [0,\tilde{t}_1)$. But this implies that $\chi(\tilde{t}_1)=\lim_{n\to\infty}\chi_n(t_1(n))=x_i$, contradicting the properties of $t_1$. Therefore $t_1(n)\uparrow t_1$ and convergence on $[0,t_1]$. The repetition of the argument for $[t_1,t_2)$ is analogous to the previous case. \\
To finish, we claim that $t_n\uparrow \infty$ as $n\to\infty$ for almost all $\omega$ and therefore the sequence eventually surpasses the arbitrarily chosen $T>0$. This is due to $\PP(N_T=\infty)=0$, the assumption that the partition $(x_i)_{i\in\mathbb{Z}}$ has no accumulation points in $(0,\infty)$ and the assumption ensuring that the path cannot drift to $0$ in finite time. At this point we can conlude that $\chi_n(t)\to \chi(t)$ for all $t\in[0,T]$. But noticing the iterative nature of our argument, evident from \eqref{iteration}, we can even conclude that the convergence holds uniformly on $[0,T]$. This finishes the proof.
\end{proof}

\subsection{Ergodicity} \label{ergodicity}

Having proved existence and uniqueness of the invariant distribution $\bm{\pi}$, it remains to show that this target distribution is actually attained by $X$ as $t\to\infty$. According to \cite[Thm. 6.1]{meyn-stability-1993}, a positive Harris recurrent Markov process with unique invariant measure is $1$-ergodic if and only if some $\Delta$-skeleton chain of the Markov process is irreducible. For a process $(X_t)_{t\geq 0}$, the \emph{$\Delta$-skeleton chain} is defined by $(X_{n\Delta})_{n\in\mathbb{N}}$, $\Delta >0$, and it is called \emph{$\varphi$-irreducible} if there exists a $\sigma$-finite measure $\varphi$ s.t. for all $x\in (0,\infty)$ and $B\in\mathcal{B}((0,\infty))$ with $\varphi(B)>0$ there exists $m\in\mathbb{N}$ s.t.
\begin{equation*}
	\PP^x(X_{m\Delta}\in B) > 0.
\end{equation*}
We call $\varphi$ the irreducibility measure and call the $\Delta$-skeleton chain just \emph{irreducible} if it is $\varphi$-irreducible for some $\varphi$.

\begin{lemma}[Irreducibility of $1$-chain] \label{irreducible 1-chain half line}
Let the assumptions in Section \ref{assumptions} hold. Further assume that the Lévy measure $\mu$ fulfills Condition 1. of Theorem \ref{sampling}. Then the $1$-skeleton chain $(X_n)_{n
\in \NN}$ is irreducible with irreducibility measure $\lambda$, the Lebesgue measure.
\end{lemma}
\begin{proof}
	This can be proven in analogy to \cite[Lemma 3.6]{oechsler-levy-2024}. It essentially holds due to the property of the Lévy measure to be supported on every open interval.
\end{proof}

In practice, the speed of ergodicity is decisive, as discussed in the introduction. According to \cite[Thm. 6.1]{meyn-stability-1993-1}, the following strengthened condition on the pointwise generator implies exponential $h$-ergodicity of the solution of \eqref{levy langevin diffusion} to the invariant measure, given that all compact sets are petite for some skeleton chain of the Markov process.

\begin{condition}[Exponential ergodicity] \label{exp recurrence condition}
There exist constants $c>0,d<\infty$, and a norm-like function $h$ such that
\begin{equation*}
\mathcal{A}h(x)\leq -ch(x) + d, \quad x\in(0,\infty).
\end{equation*}
\end{condition}

Since in Lemma \ref{compacts petite} we have seen that all compact sets are petite for $(X_t)_{t\geq 0}$, it remains to show that there exists some $\Delta$-skeleton chain s.t. all compact sets are petite for this chain. Indeed the results so far imply that this holds for any skeleton chain which is an immediate consequence of \cite[Prop. 6.1]{meyn-stability-1993}. We are now in the position to close this section with the proof of Theorem \ref{sampling}.
\begin{proof}[Proof of Theorem \ref{sampling}]
According to Theorem \ref{distributional equation}, we have constructed $X=(X_t)_{t\geq 0}$ s.t. the target distribution $\bm{\pi}$ is infinitesimally invariant for $X$. We recall that already under the assumptions made in Section \ref{assumptions} all compact sets are petite for the strong solution to \eqref{levy langevin diffusion}, as shown in Lemma \ref{compacts petite}. Note that this is independent of the specific assumptions made in Theorem~\ref{sampling} w.r.t. the target. Further, according to Theorem \ref{ergodicity}, the assumptions made in Theorem \ref{sampling} w.r.t. the Lévy measure imply that $X$ is positive Harris recurrent, which then implies the existence of a unique invariant probability measure, cf. \cite[Thm. I.3]{azema-mesure-1967}.

In a next step, in Theorem \ref{uniqueness invariant measure} and Theorem \ref{approximation} we have shown that this invariant probability measure actually agrees with the infinitesimally invariant probability measure $\bm{\pi}$. As pointed out above, at this point we can conclude that $X$ is 1-ergodic towards $\bm{\pi}$ due to \cite[Thm. 6.1]{meyn-stability-1993} whose conditions are complete with Lemma \ref{irreducible 1-chain half line} and the fact that Condition~\ref{recurrence condition} holds. 

In a final step, a close look to the proof of Theorem \ref{sampling} reveals that Condition \ref{recurrence condition} is in fact fulfilled with the particular function $g=h$, which amounts to Condition \ref{exp recurrence condition} being true for $X$. In particular, this holds since our estimate remains true by substitution $d\mathds{1}_C$ with $d$. 
\end{proof}

\section{Conclusion and outlook}

Our method successfully produces samples from heavy-tailed target distributions. In particular, due to the choice of a Lévy measure with tail behaviour matching the target distributions tail, it not only achieves exponential ergodicity in case of regularly varying target distributions, but also for heavy-tailed target distributions with tails lighter and heavier than Pareto tails. This expands the class of heavy tailed distributions that can be sampled from in exponential speed handled in \cite{huang-approximation-2020} and \cite{zhang-ergodicity-2023}. 

As mentioned in the introduction, apart from being less demanding in terms of smoothness of the target density, the method of this work is also advantageous in two aspects in comparison to the aforementioned publications. First, the choice of a compound Poisson process, in contrast to a stable process, allows for an exact simulation of the driving noise, while the infinite intensity of a stable process always needs a truncation in order to simulate it. Second, the drift coefficient given by a fractional Laplacian in \cite{huang-approximation-2020} and \cite{zhang-ergodicity-2023} probably is harder to approximate in an implementation in comparison to the drift coefficient $\phi$ in the present work, which is accessible by Monte Carlo methods due to the finite-intensity jumps.

Last, and so far unmentioned, the simple form of SDE \eqref{levy langevin diffusion} in case of a compound Poisson noise, allows for the numerical scheme that consists of simulating the jump times of the Poisson process associated to $L_t$ and solve the (deterministic) autonomous differential equation in between jumps. The analysis of the numerical error introduced by approximating the solution to \eqref{levy langevin diffusion} should therefore be favorable in our setting and is one of the main future directions of this work.

To finish, as we mentioned in the introductory section, MCMC methods obtain their right to existence from their superior performance in multivariate problems. We therefore identify the generalisation to the multivariate case as the other main direction of future work. This is intertwined with the aforementioned numerical approximation, especially due to the necessity to additionally approximate the multivariate integral appearing in $\phi$.

\section*{Statements and declarations}

\subsection*{Acknowledgement}

We thank the anonymous referee for their time and valuable comments, in particular those that led to a clearer formulation of the main theorem.

\subsection*{Funding}

The authors acknowledge financial support from the Federal Ministry of Education and Research of Germany and from the Sächsisches Staatsministerium für Wissenschaft, Kultur und Tourismus under the program Center of Excellence for AI Research, “Center for Scalable Data Analytics and Artificial Intelligence Dresden/Leipzig”, project identification number: ScaDS.AI.

\end{document}